\documentclass[12pt]{amsart}
\usepackage{mathrsfs}
\usepackage{etex}       
\usepackage{txfonts}
\usepackage{amssymb}
\usepackage{diagbox}
\usepackage{eucal}
\usepackage{graphicx}
\usepackage{amsmath}
\usepackage{amscd}
\usepackage[all]{xy}           
\usepackage{tikz,color,xcolor}
\usepackage{amsfonts,latexsym}
\usepackage{xspace}
\usepackage{epsfig}
\usepackage{float}
\usepackage{fancybox}
\usepackage{multicol}
\usepackage{colordvi}
\usepackage{makecell}
\usepackage{booktabs}
\usepackage{cancel}
\usepackage{pstricks-add} 
\usepackage{arydshln}
\usepackage{caption}
\usepackage[normalem]{ulem}
\usepackage[active]{srcltx} 
\ifpdf
\usepackage[colorlinks,final,backref=page,hyperindex]{hyperref}
\else
\usepackage[colorlinks,final,backref=page,hyperindex,hypertex]{hyperref}
\fi
\usepackage{tikz}
\usetikzlibrary{arrows.meta, positioning}
\usetikzlibrary{arrows.meta, positioning, calc}
\newtheorem{theorem}{Theorem}[section]
\newtheorem{prop}[theorem]{Proposition}
\newtheorem{lemma}[theorem]{Lemma}
\newtheorem{coro}[theorem]{Corollary}

\theoremstyle{definition}
\newtheorem{defn}[theorem]{Definition}
\newtheorem{remark}[theorem]{Remark}
\newtheorem{exam}[theorem]{Example}

\newtheorem{prop-def}[theorem]{Proposition-Definition}
\newtheorem{coro-def}[theorem]{Corollary-Definition}

\newcommand{\nc}{\newcommand}
\nc{\tred}[1]{\textcolor{red}{#1}}
\nc{\tblue}[1]{\textcolor{blue}{#1}}
\nc{\tgreen}[1]{\textcolor{green}{#1}}
\nc{\tpurple}[1]{\textcolor{purple}{#1}}
\nc{\btred}[1]{\textcolor{red}{\bf #1}}
\nc{\btblue}[1]{\textcolor{blue}{\bf #1}}
\nc{\btgreen}[1]{\textcolor{green}{\bf #1}}
\nc{\btpurple}[1]{\textcolor{purple}{\bf #1}}
\nc{\NN}{{\mathbb N}}

\nc{\vsa}{\vspace{-.1cm}} \nc{\vsb}{\vspace{-.2cm}}
\nc{\vsc}{\vspace{-.3cm}} \nc{\vsd}{\vspace{-.4cm}}
\nc{\vse}{\vspace{-.5cm}}

\renewcommand{\textbf}[1]{}

\newcommand{\delete}[1]{}

\nc{\mlabel}[1]{\label{#1}}  
\nc{\mcite}[1]{\cite{#1}}  
\nc{\mref}[1]{\ref{#1}}  
\nc{\meqref}[1]{\eqref{#1}}  
\nc{\mbibitem}[1]{\bibitem{#1}} 

\delete{
\nc{\mlabel}[1]{\label{#1}  
{\hfill \hspace{1cm}{\tt{{\ }\hfill(#1)}}}}
\nc{\mcite}[1]{\cite{#1}{{\tt{{\ }(#1)}}}}  
\nc{\mref}[1]{\ref{#1}{{\tt{{\ }(#1)}}}}  
\nc{\meqref}[1]{\eqref{#1}{{\tt{{\ }(#1)}}}}  
\nc{\mbibitem}[1]{\bibitem[\bf #1]{#1}} 
}

\nc{\opa}{\ast} \nc{\opb}{\odot}  \nc{\pa}{\frakL}
\nc{\arr}{\rightarrow} \nc{\lu}[1]{(#1)} \nc{\mult}{\mrm{mult}}
\nc{\diff}{\mathfrak{Diff}}
\nc{\opc}{\sharp}\nc{\opd}{\natural}
\nc{\dpt}{\mathrm{d}}
\nc{\tforall}{\text{ for all }}
\nc{\diam}{alternating\xspace}
\nc{\Diam}{Alternating\xspace}
\nc{\cdiam}{alternating\xspace}
\nc{\Cdiam}{Alternating\xspace}
\nc{\AW}{\mathcal{A}}
\nc{\rba}{Rota-Baxter algebra\xspace}

\nc{\ari}{\mathrm{ar}}

\nc{\lef}{\mathrm{lef}}

\nc{\Sh}{\mathrm{ST}}

\nc{\Cr}{\mathrm{Cr}}

\nc{\st}{{Schr\"oder tree}\xspace}
\nc{\sts}{{Schr\"oder trees}\xspace}

\nc{\vertset}{\Omega} 

\delete{
\nc{\assop}{\quad \begin{picture}(5,5)(0,0)
\line(-1,1){10}
\put(-2.2,-2.2){$\ob$}
\line(0,-1){10}\line(1,1){10}
\end{picture} \quad \smallskip}

\nc{\operator}{\begin{picture}(5,5)(0,0)
\line(0,-1){6}
\put(-2.6,-1.8){$\bullet$}
\line(0,1){9}
\end{picture}}

\nc{\idx}{\begin{picture}(6,6)(-3,-3)
\put(0,0){\line(0,1){6}}
\put(0,0){\line(0,-1){6}}
 \end{picture}}
}

\nc{\pb}{{\mathrm{pb}}}
\nc{\Lf}{{\mathrm{Lf}}}

\nc{\lft}{{left tree}\xspace}
\nc{\lfts}{{left trees}\xspace}

\nc{\fat}{{fundamental averaging tree}\xspace}

\nc{\fats}{{fundamental averaging trees}\xspace}
\nc{\avt}{\mathrm{Avt}}

\nc{\rass}{{\mathit{RAss}}}

\nc{\aass}{{\mathit{AAss}}}

\nc{\vin}{{\mathrm Vin}}    
\nc{\lin}{{\mathrm Lin}}    
\nc{\inv}{\mathrm{I}n}
\nc{\gensp}{V} 
\nc{\genbas}{\mathcal{V}} 
\nc{\bvp}{V_P}     
\nc{\gop}{{\,\omega\,}}     

\nc{\bin}[2]{ (_{\stackrel{\scs{#1}}{\scs{#2}}})}  
\nc{\binc}[2]{ \left (\!\! \begin{array}{c} \scs{#1}\\
    \scs{#2} \end{array}\!\! \right )}  
\nc{\bincc}[2]{  \left ( {\scs{#1} \atop
    \vspace{-1cm}\scs{#2}} \right )}  
\nc{\bs}{\bar{S}} \nc{\cosum}{\sqsubset} \nc{\la}{\longrightarrow}
\nc{\rar}{\rightarrow} \nc{\dar}{\downarrow} \nc{\dprod}{**}
\nc{\dap}[1]{\downarrow \rlap{$\scriptstyle{#1}$}}
\nc{\md}{\mathrm{dth}} \nc{\uap}[1]{\uparrow
\rlap{$\scriptstyle{#1}$}} \nc{\defeq}{\stackrel{\rm def}{=}}
\nc{\disp}[1]{\displaystyle{#1}} \nc{\dotcup}{\
\displaystyle{\bigcup^\bullet}\ } \nc{\gzeta}{\bar{\zeta}}
\nc{\hcm}{\ \hat{,}\ } \nc{\hts}{\hat{\otimes}}
\nc{\barot}{{\otimes}} \nc{\free}[1]{\bar{#1}}
\nc{\uni}[1]{\tilde{#1}} \nc{\hcirc}{\hat{\circ}} \nc{\lleft}{[}
\nc{\lright}{]} \nc{\lc}{\lfloor} \nc{\rc}{\rfloor}
\nc{\curlyl}{\left \{ \begin{array}{c} {} \\ {} \end{array}
    \right .  \!\!\!\!\!\!\!}
\nc{\curlyr}{ \!\!\!\!\!\!\!
    \left . \begin{array}{c} {} \\ {} \end{array}
    \right \} }
\nc{\longmid}{\left | \begin{array}{c} {} \\ {} \end{array}
    \right . \!\!\!\!\!\!\!}
\nc{\onetree}{\bullet} \nc{\ora}[1]{\stackrel{#1}{\rar}}
\nc{\ola}[1]{\stackrel{#1}{\la}}
\nc{\ot}{\otimes} \nc{\mot}{{{\boxtimes\,}}}
\nc{\otm}{\overline{\boxtimes}} \nc{\sprod}{\bullet}
\nc{\scs}[1]{\scriptstyle{#1}} \nc{\mrm}[1]{{\rm #1}}
\nc{\margin}[1]{\marginpar{\rm #1}}   
\nc{\dirlim}{\displaystyle{\lim_{\longrightarrow}}\,}
\nc{\invlim}{\displaystyle{\lim_{\longleftarrow}}\,}
\nc{\mvp}{\vspace{0.3cm}} \nc{\tk}{^{(k)}} \nc{\tp}{^\prime}
\nc{\ttp}{^{\prime\prime}} \nc{\svp}{\vspace{2cm}}
\nc{\vp}{\vspace{8cm}} \nc{\proofbegin}{\noindent{\bf Proof: }}
\nc{\proofend}{$\blacksquare$ \vspace{0.3cm}}
\nc{\modg}[1]{\!<\!\!{#1}\!\!>}
\nc{\intg}[1]{F_C(#1)} \nc{\lmodg}{\!
<\!\!} \nc{\rmodg}{\!\!>\!}
\nc{\cpi}{\widehat{\Pi}}
\nc{\sha}{{\mbox{\cyr X}}}  
\nc{\shap}{{\mbox{\cyrs X}}} 
\nc{\shan}{{\overrightarrow \sha}}
\nc{\shpr}{\diamond}    
\nc{\shp}{\ast} \nc{\shplus}{\shpr^+}
\nc{\shprc}{\shpr_c}    
\nc{\msh}{\ast} \nc{\zprod}{m_0} \nc{\oprod}{m_1}
\nc{\vep}{\varepsilon} \nc{\labs}{\mid\!} \nc{\rabs}{\!\mid}
\nc{\sqmon}[1]{\langle #1\rangle}

\nc{\mmbox}[1]{\mbox{\ #1\ }} \nc{\fp}{\mrm{FP}}
\nc{\rchar}{\mrm{char}} \nc{\End}{\mrm{End}} \nc{\Fil}{\mrm{Fil}}
\nc{\Mor}{Mor\xspace} \nc{\gmzvs}{gMZV\xspace}
\nc{\gmzv}{gMZV\xspace} \nc{\mzv}{MZV\xspace}
\nc{\mzvs}{MZVs\xspace} \nc{\Hom}{\mrm{Hom}} \nc{\id}{\mrm{id}}
\nc{\im}{\mrm{im}} \nc{\incl}{\mrm{incl}} \nc{\map}{\mrm{Map}}
\nc{\mchar}{\rm char} \nc{\nz}{\rm NZ} \nc{\supp}{\mathrm Supp}

\nc{\Alg}{\mathbf{Alg}} \nc{\Bax}{\mathbf{Bax}} \nc{\bff}{\mathbf f}
\nc{\bfk}{{\bf k}} \nc{\bfone}{{\bf 1}} \nc{\bfx}{\mathbf x}
\nc{\bfy}{\mathbf y}
\nc{\base}[1]{\bfone^{\otimes ({#1}+1)}} 
\nc{\Cat}{\mathbf{Cat}}

\nc{\detail}{\marginpar{\bf More detail}
    \noindent{\bf Need more detail!}
    \svp}
\nc{\Int}{\mathbf{Int}} \nc{\Mon}{\mathbf{Mon}}
\nc{\rbtm}{{shuffle }} \nc{\rbto}{{Rota-Baxter }}
\nc{\remarks}{\noindent{\bf Remarks: }} \nc{\Rings}{\mathbf{Rings}}
\nc{\Sets}{\mathbf{Sets}} \nc{\wtot}{\widetilde{\odot}}
\nc{\wast}{\widetilde{\ast}} \nc{\bodot}{\bar{\odot}}
\nc{\bast}{\bar{\ast}} \nc{\hodot}[1]{\odot^{#1}}
\nc{\hast}[1]{\ast^{#1}} \nc{\mal}{\mathcal{O}}
\nc{\tet}{\tilde{\ast}} \nc{\teot}{\tilde{\odot}}
\nc{\oex}{\overline{x}} \nc{\oey}{\overline{y}}
\nc{\oez}{\overline{z}} \nc{\oef}{\overline{f}}
\nc{\oea}{\overline{a}} \nc{\oeb}{\overline{b}}
\nc{\weast}[1]{\widetilde{\ast}^{#1}}
\nc{\weodot}[1]{\widetilde{\odot}^{#1}} \nc{\hstar}[1]{\star^{#1}}
\nc{\lae}{\langle} \nc{\rae}{\rangle}
\nc{\lf}{\lfloor}
\nc{\rf}{\rfloor}

\nc{\QQ}{{\mathbb Q}}
\nc{\RR}{{\mathbb R}} \nc{\ZZ}{{\mathbb Z}}
\nc{\CC}{{\mathbb C}}

\nc{\cala}{{\mathcal A}} \nc{\calb}{{\mathcal B}}
\nc{\calc}{{\mathcal C}}
\nc{\cald}{{\mathcal D}} \nc{\cale}{{\mathcal E}}
\nc{\calf}{{\mathcal F}} \nc{\calg}{{\mathcal G}}
\nc{\calh}{{\mathcal H}} \nc{\cali}{{\mathcal I}}
\nc{\call}{{\mathcal L}} \nc{\calm}{{\mathcal M}}
\nc{\caln}{{\mathcal N}}\nc{\calo}{{\mathcal O}}
\nc{\calp}{{\mathcal P}} \nc{\calq}{\mathcal{Q}} \nc{\calr}{{\mathcal R}}
\nc{\cals}{{\mathcal S}} \nc{\calt}{{\mathcal T}}
\nc{\calu}{{\mathcal U}} \nc{\calw}{{\mathcal W}} \nc{\calk}{{\mathcal K}}
\nc{\calx}{{\mathcal X}} \nc{\CA}{\mathcal{A}}

\nc{\fraka}{{\mathfrak a}} \nc{\frakA}{{\mathfrak A}}
\nc{\frakb}{{\mathfrak b}} \nc{\frakB}{{\mathfrak B}}
\nc{\frakD}{{\mathfrak D}} \nc{\frakF}{\mathfrak{F}}
\nc{\frakf}{{\mathfrak f}} \nc{\frakg}{{\mathfrak g}}
\nc{\frakH}{{\mathfrak H}} \nc{\frakL}{{\mathfrak L}}
\nc{\frakM}{{\mathfrak M}} \nc{\bfrakM}{\overline{\frakM}}
\nc{\frakm}{{\mathfrak m}} \nc{\frakP}{{\mathfrak P}}
\nc{\frakN}{{\mathfrak N}} \nc{\frakp}{{\mathfrak p}}
\nc{\frakS}{{\mathfrak S}} \nc{\frakT}{\mathfrak{T}}
\nc{\frakX}{{\mathfrak X}} \nc{\frakx}{\mathfrak{x}}
\nc{\frakc}{{\mathfrak c}}
\nc{\frakd}{{\mathfrak d}}
\nc{\frake}{{\mathfrak e}}
\nc{\BS}{\mathbb{S}}

\font\cyr=wncyr10 \font\cyrs=wncyr7

\nc{\ID}{\mathfrak{I}} \nc{\lbar}[1]{\overline{#1}}
\nc{\bre}{{\rm b}}
\nc{\dep}{\mrm{d}}
\nc{\sd}{\cals} \nc{\rb}{\rm RB}
\nc{\A}{\rm angularly decorated\xspace} \nc{\LL}{\rm L}
\nc{\w}{\rm wid} \nc{\arro}[1]{#1}
\nc{\ver}{\rm ver}
\nc{\dd}{\diamond}
\nc{\dr}{\diamond_r}
\nc{\dg}{{\diamond_\frake}}
\nc{\dk}{\diamond_\bfk}
\nc{\shar}{{\mbox{\cyrs X}}_r} 
\nc{\shag}{{\mbox{\cyrs X}}_g}
\nc{\de}{\Delta}
\nc{\delg}{\Delta_\frake}
\nc{\da}{\Delta_A}
\nc{\dgg}{\Delta_g}
\nc{\va}{\vep_A }
\nc{\ve}{\vep }
\nc{\vg}{\vep_\frake }
\nc{\vgg}{\vep_g }
\nc{\bug}{\bullet_e}
\nc{\bt}{\bar{\ot}}
\nc{\pe}{{P_\frake}}
\nc{\mg}{\mu_e}
\nc{\fs}{\frakS}

\nc{\mapmonoid}{\frakM}
\nc{\ncsha}{{\mbox{\cyr X}^{\mathrm NC}}} 
\nc{\ncshao}{{\mbox{\cyr X}^{\mathrm NC}}}
\nc{\ce}{{\mbox{\cyr X}_\frake^{\mathrm NC}}(A)}
\nc{\ced}{{\mbox{\cyr X}_\frake^{\mathrm NC}}(T^+(D))}
\nc{\cet}{{\mbox{\cyr X}_\frake^{\mathrm NC}}(T)}
\nc{\dfgen}{V} \nc{\dfrel}{\Lambda}
\nc{\dfgenb}{\vec{v}} \nc{\dfrelb}{\vec{r}}
\nc{\dfgene}{v} \nc{\dfrele}{r}
\nc{\dfop}{\odot}
\nc{\dfoa}{\dfop^{(1)}} \nc{\dfob}{\dfop^{(2)}}
\nc{\dfoc}{\dfop^{(3)}} \nc{\dfod}{\dfop^{(4)}}
\nc{\erba}{\mathbf{ERBA}}
\nc{\ed}{\mathbf{ED}}
\nc{\etd}{\mathbf{ETD}}
\nc{\NS}{\mathbf{NS}}
\nc{\FN}{F_{\mathrm N}}
\nc{\ob}{\ \begin{picture}(-1,1)(-1,-3)\circle*{3}\end{picture}\ \,}
\nc{\oc}{\circ}
\nc{\obp}{{\ \begin{picture}(-1,1)(-1,-3)\circle*{3}\end{picture}\  }_P}
\nc{\ocp}{\circ_P}
\nc{\pr}{\prec}
\nc{\su}{\succ}
\nc{\prp}{\prec_P}
\nc{\scp}{\succ_P}
\nc{\UN}{U_{N}}
\nc{\denshpr}{\den{\shpr}}
\nc{\den}[1]{\check{#1}}
\nc{\freea}[1]{\tilde{#1}}
\nc{\freev}[1]{\hat{#1}}
\nc{\speciall}{\mathrm{sl(2,\mathbb{C})}}

\nc{\wmrb}{extended Rota-Baxter algebra\xspace}
\nc{\wmrbs}{extended Rota-Baxter algebras\xspace}
\nc{\wmrbo}{extended Rota-Baxter operator\xspace}
\nc{\wmrbos}{extended Rota-Baxter operators\xspace}
\nc{\Wmrbos}{Extended Rota-Baxter operators\xspace}
\nc{\wmrbi}{extended Rota-Baxter identity\xspace}
\nc{\edr}{extended dendriform algebra\xspace}
\nc{\edrs}{extended dendriform algebras\xspace}
\nc{\et}{extended tridendriform algebra\xspace}
\nc{\ets}{extended tridendriform algebras\xspace}
\nc{\eqa}{extended quadri-algebra\xspace}
\nc{\eqas}{extended quadri-algebras\xspace}

\nc{\dft}{\star}
\nc{\nea}{\nearrow} 
\nc{\se}{\searrow}
\nc{\sw}{\swarrow}
\nc{\nw}{\nwarrow}
\nc{\nep}{\ne^{op}}
\nc{\sep}{\se^{op}}
\nc{\swp}{\sw^{op}}
\nc{\nwp}{\nw^{op}}
\nc{\east}{\succ}
\nc{\west}{\prec}
\nc{\north}{\wedge}
\nc{\south}{\vee}
\nc{\eastt}{\east^t}
\nc{\westt}{\west^t}
\nc{\northt}{\north^t}
\nc{\trir}{\triangleright}
\nc{\tril}{\triangleleft}

\nc{\Loday}{\mrm{binary quadratic nonsymmetric}\xspace}
\nc{\dftimes}{\blacksquare}

\nc{\oeq}[1]{\stackrel{(#1)}{=}}
\nc{\dfpair}[2]{\left[
	{{#1}\atop{#2}}\right]} 

\nc{\oerb}{\mathcal{ERB}} 
\nc{\oetd}{\mathcal{ETD}} 
\nc{\oed}{\mathcal{ED}}

\nc{\genpre}{{\text{DC}}(W)}
\nc{\morpre}{\Psi}
\nc{\relpre}{\text{DC}(S)}
\nc{\oppre}{\text{DC}(\calq)}
\nc{\genpost}{\text{TC}(W)}
\nc{\morpost}{\Phi}
\nc{\relpost}{\text{TC}(S)}
\nc{\oppost}{\text{TC}(\calq)}

\nc{\li}[1]{\textcolor{red}{#1}}
\nc{\lir}[1]{\textcolor{red}{Li:#1}}

\nc{\sh}[1]{\textcolor{blue}{Shanghua: #1}}
\nc{\hsy}[1]{\textcolor{orange}{Hsy: #1}}

\definecolor{darkred}{rgb}{0.7,0,0} 

\definecolor{darkgreen}{RGB}{0,180,0}

\begin{document}

\title[Extended tridendriform algebras and post-Lie algebras as companion structures]{Extended (tri)dendriform algebras, pre-Lie algebras and post-Lie algebras as companion structures of extended Rota-Baxter algebras}
%

\author{Shanghua Zheng}
\address{School of Mathematics and Statistics, Jiangxi Normal University, Nanchang, Jiangxi 330022, China}
\email{zhengsh@jxnu.edu.cn}

\author{Shiyu Huang}
\address{School of Mathematics and Statistics, Jiangxi Normal University, Nanchang, Jiangxi 330022, China}
\email{2640207650@qq.com}

\author{Li Guo}
\address{Department of Mathematics and Computer Science,
	Rutgers University,
	Newark, NJ 07102, USA}
\email{liguo@rutgers.edu}

\date{\today}
\begin{abstract}
Under the common theme of splitting of operations, the notions of (tri)dendriform algebras, pre-Lie algebras and post-Lie algebras have attracted sustained attention with broad applications. An important aspect of their studies is as the derived structures of Rota-Baxter operators on associative or Lie algebras. 
This paper introduces extended versions of (tri)dendriform algebras, pre-Lie algebras, and post-Lie algebras, establishing close relations among these new structures that generalize those among their classical counterparts. These new structures can be derived from the extended Rota-Baxter operator, which combines the standard Rota-Baxter operator and the modified Rota-Baxter operator. 
To characterize these new notions as {\it the} derived structures of extended Rota-Baxter algebras, we define the binary quadratic operad in companion with an operad with nontrivial unary operations. Then the extended (tri)dendriform algebra is shown to be the binary quadratic nonsymmetric operads in companion to the operad of extended Rota-Baxter algebras. 
As a key ingredient in achieving this and for its own right, the free extended Rota-Baxter algebra is constructed by bracketed words. 
\end{abstract}

\makeatletter
\@namedef{subjclassname@2020}{\textup{2020} Mathematics Subject Classification}
\makeatother
\subjclass[2020]{
17B38, 
17D25, 
18M65, 
17A30, 
16S10, 
16W99, 
}

\keywords{Rota-Baxter algebra, modified Rota-Baxter algebra, dendriform algebra, tridendriform algebra, pre-Lie algebra, post-Lie algebra, non-symmetric operad}

\maketitle

\vspace{-1.5cm}

\tableofcontents

\vspace{-1.5cm}

\setcounter{section}{0}

\allowdisplaybreaks

\section{Introduction}
This paper first introduces several nonassociative algebras generalizing the classical notions of (tri)dendriform algebras, pre-Lie algebras and post-Lie algebras. They can be derived from extended Rota-Baxter algebras. The paper then introduces a general notion of the binary quadratic operad in companion with a unary binary operad. Then extended (tri)dendriform algebras are shown to be precisely the companions of extended Rota-Baxter algebras.

\subsection{(Tri)dendriform algebras, pre-Lie algebras and post-Lie algebras}
Dendriform algebras, tridendriform algebras, pre-Lie algebras and post-Lie algebras are among a class of nonassociative algebraic structures which have been discovered at different times in the past few decades and have recently been unified under a common theme.  

The notion of a dendriform algebra was introduced by Loday~\mcite{Lo1} with motivation from algebraic $K$-theory.
A {\bf dendriform algebra} is a vector space $D$ with two binary operations $\prec$ and $\succ$ that satisfy the following relations:
	\begin{eqnarray}
		(x \prec y) \prec z &=& x \prec (y\prec z +y \succ z),\mlabel{eq:dd1}\\
		(x \succ y ) \prec z&=& x \succ (y\prec z),\mlabel{eq:dd2}  \\
		(x \prec y +x\succ y)\succ z &=& x \succ (y\succ z), \quad x, y, z\in D.
		\mlabel{eq:dd3}
	\end{eqnarray}
Introduced by Loday and Ronco  to study  polytopes and Koszul duality~\mcite{LR2}, a tridendriform algebra has three binary operations $\prec, \succ, \cdot$ satisfying seven relations. Afterwards, several generalizations and extensions of this notion have appeared~\mcite{BGN4,BR10,Cha02,Li18}. 
Dendriform algebras and tridendriform algebras share a common property,  that is, the sum of their binary operations is an associative operation. To describe this phenomena in general, the idea of splitting of the associativity was proposed by Loday~\mcite{Lo2}, and a general construction was obtained in~\mcite{BBGN,PBG}. 

As it turns out, the splitting of the Lie algebra has appeared much earlier. A {\bf pre-Lie algebra} is a vector space $A$ equipped with a binary operation $\circ$ such that, for all $x,y,z\in A$,
$$(x\circ y)\circ z-x\circ(y\circ z) =(y\circ x)\circ z-y\circ (x\circ z).$$
The commutator $[x,\,y]:=x\circ y-y\circ x$ defines a Lie bracket. Pre-Lie algebra was introduced by Gerstenhaber to 
study the deformations and cohomology theory of associative algebras~\mcite{Ge63}. Among other names, it is also called a left-symmetric algebra by Vinberg to study convex homogenous cones~\mcite{Vi63}. For surveys of this important notion, see~\mcite{Bai,Man}. 
A dendriform algebra $A$ induces a pre-Lie algebra structure on $A$ by the binary operation $x\circ y:=x\su y-y\pr x$, and so leads to a Lie algebra structure on $A$ via the commutator.  

A post-Lie algebra was introduced by Vallette in ~\cite{Va07} from his study of Koszul duality of operads and appeared in differential geometry and numerical integration on manifolds. For some of its numerous applications in mathematics, computer science and mathematical physics, see~\cite{BGN,BGST,BK,CEO,Dot,GLS21,LST}. 
A tridendriform algebra $(A,\prec,\succ,\cdot)$ gives rise to a post-Lie algebra structure on $A$ by anti-symmetrization. 
The above relations can be summarized in the following diagram.
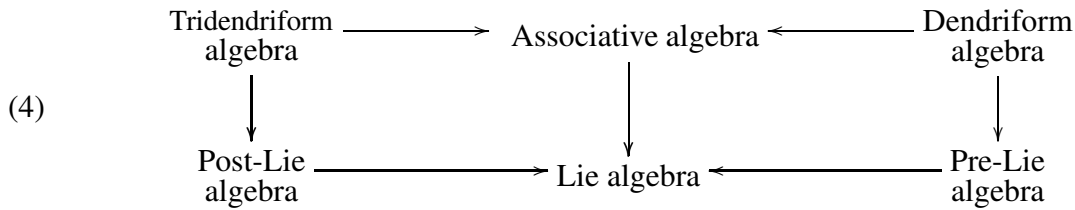
\begin{equation}
	\begin{split}
		\xymatrix{\txt{\small  Tridendriform\\ algebra} \ar[d] \ar[rr]\ar[d]&&\txt{\rm\, Associative algebra}\ar[d] \ar[d]&& \txt{\rm  Dendriform \\algebra} \ar[ll]\ar[d]\\
			\txt{\rm Post-Lie\\ algebra}\ar[rr]&&\txt{\rm  Lie algebra} 
			&& \txt{\rm Pre-Lie\\ algebra}\ar[ll]
		}
		\mlabel{eq:dcom}
	\end{split}
\end{equation}

\subsection{(Tri)dendriform algebras, pre-Lie algebras and post-Lie algebras as derived structures of Rota-Baxter algebras}
A theme that unifies the previous notions is that they can all be derived from a Rota-Baxter operator acting on an associative algebra or a Lie algebra.  

The study of Rota-Baxter algebras originated from the 1960 work of G. Baxter on the fluctuation theory in probability~\mcite{Ba}.  
For a fixed scalar $\lambda$, a {\bf Rota-Baxter algebra of weight $\lambda$} is a pair $(R, P)$ consisting of an associative algebra $R$ and a linear operator $P:R\to R$  satisfying the {\bf Rota-Baxter identity}
\begin{equation*}
	P(x)P(y)=P(xP(y))+P(P(x)y)+\lambda P(xy),\;x,y\in R.
\end{equation*}
Then $P$ is called a {\bf Rota-Baxter operator of weight $\lambda$}. 

Differing to Rota-Baxter operator by an affine transformation, the structure of modified Rota-Baxter algebras already appeared in the 1951 study of Tricomi on ergodic theory~\mcite{Tri}, and was again developed by Cotlar in 1955~\mcite{Co}. Here a linear operator  $Q$ is called a {\bf modified Rota-Baxter operator of weight $\kappa$} if 
\begin{equation*}
	Q(x)Q(y)=Q(xQ(y))+Q(Q(x)y)+\kappa xy,\;x,y\in R.
\end{equation*}
Then $(R,Q)$ is called a {\bf modified Rota-Baxter algebra of weight $\kappa$}. Recent studies of modified Rota-Baxter algebras can be found in~\mcite{JS24,Li18,ZGG191,ZGG192,ZZG23}. 

As a remarkable coincidence, their Lie algebra analogs were independently discovered by Semenov-Tian-Shansky in 1983~\mcite{STS} as the operator forms of the classical Yang-Baxter equation. These classical operators have attracted renewed interests in recent years thanks to their broad applications to areas including renormalization
of quantum field theory~\mcite{CK00}, Yang-Baxter equations~\mcite{BGN,BGN2}, pre-Lie algebras~\mcite{Bai,Cha01}, operads~\mcite{Ag1,BBGN,WZ}, cohomology and deformations~\mcite{TBGS}, combinatorics~\mcite{YGT} and Lie groups \mcite{BGST,GLS21}.

A distinct aspect of the renewed interest on Rota-Baxter algebras comes from their induced structures.
Aguiar~\mcite{Ag1} first discovered that a Rota-Baxter algebra $(A,P)$ of weight $0$  carries a  dendriform algebra structure on $A$ given by
\begin{equation}\mlabel{eq:rbd}
	x\prec_P y:=xP(y),\quad x\succ_P y:=P(x)y,\quad x,y\in A.
\end{equation}
The connection between Rota-Baxter algebras of weight $\lambda$ and tridendriform algebras~\mcite{EF} was also obtained.
Let $P$ be a Rota-Baxter operator of weight $0$ (resp. $1$) on a Lie algebra $\frakg$. Then the binary operation 
$x\circ y:=[P(x),y]$
gives rise to a pre-Lie algebra (resp. post-Lie algebra) structure on $\frakg$~\mcite{BGN,GS}.
Such connections were pursued further and established broadly for the splittings of operads~\mcite{BBGN,Dot,Gu18,PBG}.

\subsection{Derived structures of extended Rota-Baxter algebras} 
As a common generalization of the Rota-Baxter algebra and modified Rota-Baxter algebra, extended Rota-Baxter algebras were introduced and studied systematically, motivated by their applications to classical and associative Yang-Baxter equations~\mcite{BGN,BGN2,ZGQ}.
This provides a broader context and further motivation to study the induced structures of extended Rota-Baxter algebras and in particular of modified Rota-Baxter algebras. This is the purpose of this paper. 

More precisely, the purpose of this paper is two-fold, summarized in the following two paragraphs, in completion of each other. The first part, consisting of Section~\mref{sec:sdn}, introduces extended versions of (tri)dendriform algebras and pre-/post-Lie algebras, and shows that they are derived from extended Rota-Baxter associative and Lie algebras, respectively. The second part, comprised of Sections~\mref{sec:FERBA} and \mref{sec:wmrbet}, establishes a general framework to define {\em the} (binary quadratic) companion structure of an associative algebra equipped with linear operators. This framework allows us to show that the structures introduced in the first parts are not only derived from the extended Rota-Baxter associative or Lie algebras --- they are the companion structures of the extended Rota-Baxter associative or Lie algebras.  

\subsubsection{Extended dendriform type algebras and their derivations from extended Rota-Baxter algebras} 
In the first part of the paper (Section~\mref{sec:sdn}), we introduce the notions of \edrs and \ets generalizing the notions of dendriform algebras and tridendriform algebras. We also introduce the notions of extended pre-Lie algebras and extended post-Lie algebras, generalizing the notions of pre-Lie algebras and post-Lie algebras. The interconnections among these structures are established, extending the diagram in Eq.~\meqref{eq:dcom} to the following commutative diagram (Theorem~\mref{thm:edcom}).
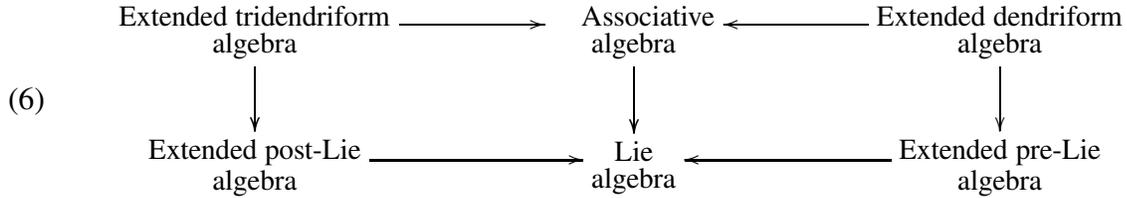
\begin{equation}
{\small	\begin{split}
		\xymatrix{\txt{\rm  Extended tridendriform\\ algebra}\ar[d] \ar@/2pc/[rr]\ar[d]&&\txt{\rm\quad Associative\\ algebra}\ar[d] \ar[d]&& \txt{\rm  Extended dendriform \\algebra} \ar@/2pc/[ll]\ar[d]
			\\
			\txt{\rm Extended post-Lie\\ algebra}\ar[rr]&&\txt{\rm  Lie \\algebra} 
			&& \txt{\rm Extended pre-Lie\\ algebra}\ar[ll]}
		\mlabel{eq:edcom}
	\end{split}
}
\end{equation}

In fact, the extended tridendriform algebras and extended post-Lie algebras are double weighted just like extended Rota-Baxter associative algebras and Lie algebras. These general notions not only unify several previous notions, they are also shown to be derived from extended Rota-Baxter operators on associative algebras and Lie algebras respectively, generalizing the classical relations~\mcite{Ag1,BGN,EF,GS}, as depicted in the following diagram. Here the newly introduced notions are in boldface.  

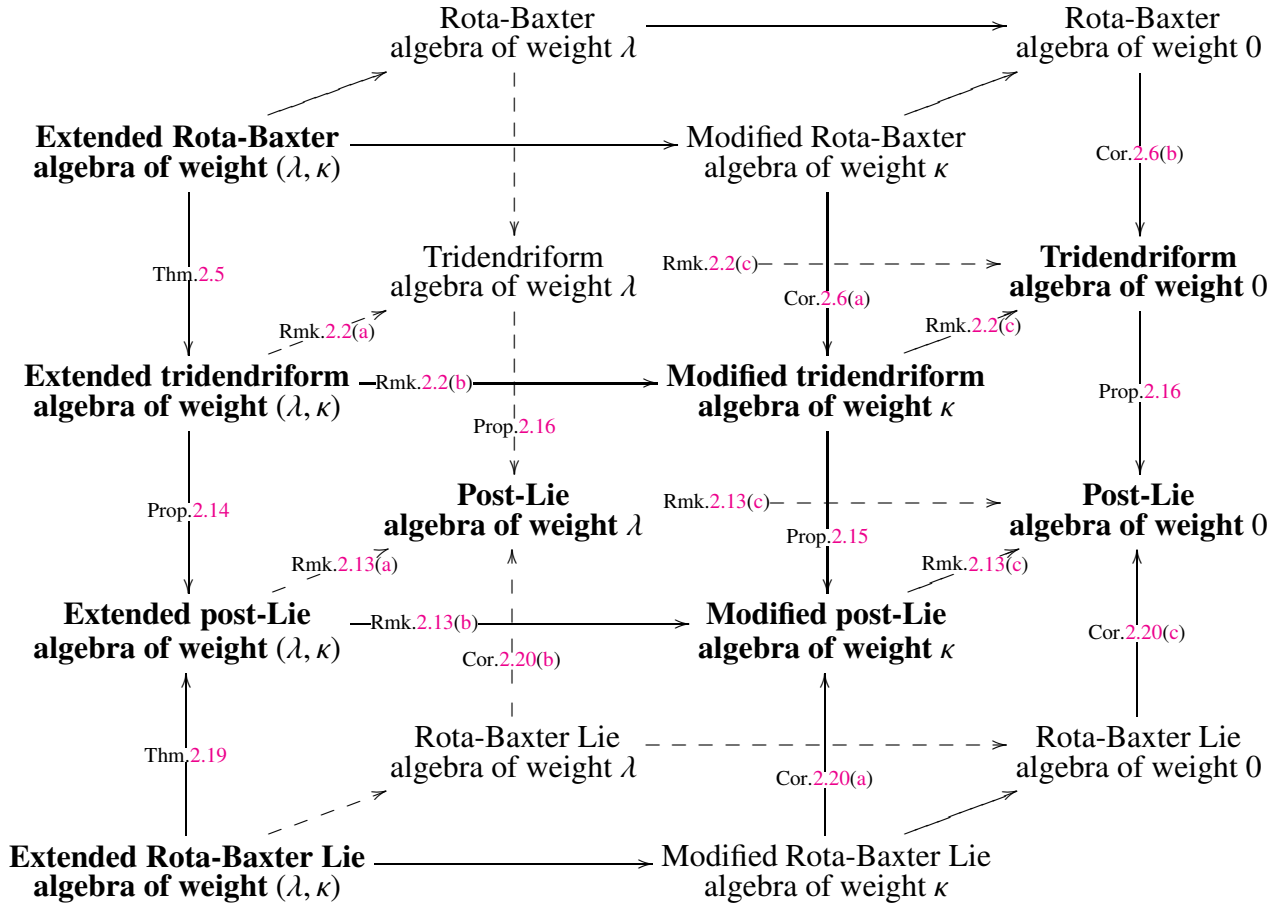
\xymatrix@C=0.001ex@R=3ex{
	& \txt{ Rota-Baxter\\ algebra of weight $\lambda$}  \ar@<.1ex>@{-->}[dd]  \ar@<.1ex>@{->}[rr]
	& &\txt{ Rota-Baxter\\ algebra of weight $0$}  \ar@<.1ex>@{->}[dd]|-{\rm Cor.\mref{co:merbt}\meqref{it:mo2}}
	&      \\
	\txt{\bf Extended Rota-Baxter\\ \bf algebra of weight $(\lambda,\kappa)$}  \ar@<.1ex>@{->}[dd]|-{\rm Thm.\mref{thm:erbt}}\ar@<.1ex>@{->}[rr] \ar@<.1ex>@{->}[ur]
	& &\txt{Modified Rota-Baxter\\ algebra of weight $\kappa$}  \ar@<.1ex>@{->}[dd]|------>{\rm Cor.\mref{co:merbt}\meqref{it:mo1}}   \ar@<.1ex>@{->}[ur]
	& &    \\
	&  \txt{Tridendriform\\ algebra of weight $\lambda$ }    \ar@<.1ex>@{-->}[dd]|----->{\rm Prop.\mref{prop:triap}}  \ar@<.1ex>@{-->}[rr]|-------<{\rm Rmk.\mref{rk:cases}\meqref{it:rk3}}
	&  & \txt{\bf Tridendriform \\\bf algebra of weight $0$}  \ar@<.1ex>@{->}[dd]|-{\rm Prop.\mref{prop:triap}} 
	&     \\
	\txt{\bf Extended tridendriform \\ \bf algebra of weight $(\lambda,\kappa)$} \ar@<.1ex>@{->}[dd]|-{\rm Prop.\mref{prop:etap}}  \ar@<.1ex>@{->}[rr]|-------<{\rm Rmk.\mref{rk:cases}\meqref{it:rk2}} \ar@<.1ex>@{-->}[ur]|-{\rm Rmk.\mref{rk:cases}\meqref{it:rk1}}
	&  & \txt{\bf Modified tridendriform\\\bf algebra of weight $\kappa$}   \ar@<.1ex>@{->}[dd]|------>{\rm Prop.\mref{prop:mtmp}} \ar@<.1ex>@{->}[ur]|----->{\rm Rmk.\mref{rk:cases}\meqref{it:rk3}}
	&  &     \\
	&  \txt{\bf Post-Lie \\\bf algebra of weight $\lambda$}    \ar@<.1ex>@{-->}[rr]|-------<{\rm Rmk.\mref{rk:epost}\meqref{it:epost3}}
	&  & \txt{\bf Post-Lie \\\bf algebra of weight $0$}   
	&     \\
	\txt{\bf Extended post-Lie \\ \bf algebra of weight $(\lambda,\kappa)$}   \ar@<.1ex>@{->}[rr]|--------<{\rm Rmk.\mref{rk:epost}\meqref{it:epost2}} \ar@<.1ex>@{-->}[ur]|----->{\rm Rmk.\mref{rk:epost}\meqref{it:epost1}}
	&  & \txt{\bf Modified post-Lie \\\bf algebra of weight $\kappa$}   \ar@<.1ex>@{->}[ur]|----->{\rm Rmk.\mref{rk:epost}\meqref{it:epost3}}
	&  & \\
	&  \txt{ Rota-Baxter Lie\\ algebra of weight $\lambda$} \ar@<.1ex>@{-->}[uu]|------<{\rm Cor.\mref{co:postl}\meqref{it:mp2}}   \ar@<.1ex>@{-->}[rr]
	&  & \txt{Rota-Baxter Lie \\algebra of weight $0$}  \ar@<.1ex>@{->}[uu]|-{\rm Cor.\mref{co:postl}\meqref{it:mp3}} 
	&     \\
	\txt{\bf Extended Rota-Baxter Lie\\ \bf algebra of weight $(\lambda,\kappa)$} \ar@<.1ex>@{->}[uu]|-{\rm Thm.\mref{prop:eepost}}  \ar@<.1ex>@{->}[rr] \ar@<.1ex>@{-->}[ur]
	&  & \txt{Modified Rota-Baxter Lie\\ algebra of weight $\kappa$}  \ar@<.1ex>@{->}[uu]|------<{\rm Cor.\mref{co:postl}\meqref{it:mp1}} \ar@<.1ex>@{->}[ur]
	&  &   }

\subsubsection{Determining all derived structures of extended Rota-Baxter algebras}
Knowing that extended (tri)dendrifrom algebras are derived from extended Rota-Baxter algebras as indicated above, 
one is naturally led to the converse question: what other quadratic nonsymmetric relations can derived this way? In fact, there have been numerous studies on algebraic structures derived from an {\bf operated algebra}~\mcite{Guo09}, defined to be a pair $(R,P)$ where $R$ can be an associative algebra or a Lie algebra or any algebra with binary operations, and $P$ is a linear operator on $R$, such as a derivation, Rota-Baxter operator, Nijenhuis operator or averaging operato. One critical yet often mysterious aspect of these studies is to obtain binary structures that are derived from the operated algebras, and furthermore to determine all such binary structuresr~\mcite{BBGN,DS24,HBG23,Le04,LLB23,ML24,PBG}. See~\mcite{KSO} for the case of derivations and~\mcite{LG} for the case of Nijenhuis (associative) algebras.

To give a precise formulation and answer to this type of questions, we first introduce the general notions of tri-companion and di-companion of the operad of an operated algebra (Definitions~\mref{de:postind} and \mref{de:preind}). These companions give all binary quadratic operads that can be possibly derived from the operad of the operated algebra. We then use the construction of free \wmrbs (Theorem~\mref{thm:free}) to show that the operads of extended tridendriform algebras and extended dendriform algebras are respectively the tri-companion and di-companion of the operad of \wmrbs (Theorems~\mref{thm:owdn} and \mref{thm:edn}). This di-companion leads to the notion and classification of more general extended dendriform algebras (Definition~\mref{defn:type}). 

\smallskip

\noindent
{\bf Convention.} Throughout this paper, $\bfk$ is a field of characteristic $0$.  A $\bfk$-algebra is taken to be
nonunitary associative unless otherwise specified.


\section{Extended (tri)dendriform algebras, extended pre-Lie algebras and extended post-Lie algebras}
\mlabel{sec:sdn}
After the introduction of tridendriform algebras by Loday and Ronco~\mcite{LR2}, several generalizations and extensions have appeared~\mcite{BGN4,BR10,Cha02,Li18}. 
We give a common generalization of these notions. Furthermore, by  simply compressing the three relations in a dendriform algebra into two relations, we obtain the notion of an \edr. Their relationships to extended Rota-Baxter algebras are also established in analog to those between dendriform algebras and tridendriform algebras and Rota-Baxter algebras.

\subsection{Extended tridendriform algebras and extended dendriform algebras}
Recall that a {\bf tridendriform algebra} \mcite{LR2} is
		a vector space $T$ equipped with three binary operations $\prec,\succ$ and
		$\ob$ that satisfy the following relations. 
			\begin{align*}
				(x\prec y)\prec z&=x\prec (y\star z), \\
				(x\succ y)\prec z&=x\succ (y\prec z), \\
				(x\star y)\succ z&=x\succ (y\succ z), \\
				(x\succ y)\ob z&=x\succ (y\ob z),\\
				(x\prec y)\ob z&=x\ob (y\succ z),\\
				(x\ob y)\prec z&=x\ob (y\prec z),\\
				(x\ob y)\ob z&=x\ob (y\ob z), \quad x, y, z\in T. 
			\end{align*}
		Here $\star=\prec+\succ+\ob$.
		
We now give our key notion generalizing the tridendriform algebra and related structures.
 
\begin{defn}\mlabel{defn:etri} Let $\lambda,\kappa \in\bfk$ be fixed. An {\bf \et  of weight $(\lambda,\kappa)$}, or simply an {\bf \et}, is a vector space $A$ with three binary operations $\pr,\su$ and $\ob$ that satisfy the following relations.
\begin{align}
(x\pr y)\pr z&= x\pr (y\star_\lambda z)+\kappa x\ob(y\ob z), \mlabel{eq:e1}\\
(x\su y)\pr z&=x\su (y\pr z), \mlabel{eq:e2}\\
(x\star_\lambda y)\su z+\kappa (x\ob y)\ob z&=x\su(y\su z), \mlabel{eq:e3}\\
(x\su y)\ob z&=x\su (y\ob z), \mlabel{eq:e4}\\
(x\pr y)\ob z&=x\ob (y\su z), \mlabel{eq:e5}\\
(x\ob y)\pr z&=x\ob (y\pr z), \mlabel{eq:e6}\\
(x\ob y)\ob z&=x\ob (y\ob z), \quad x,y,z \in A,\mlabel{eq:e7}
\end{align}
where $\star_\lambda$ is the abbreviation
\begin{equation}
	\star_\lambda:=\pr +\su +\lambda \ob,
\mlabel{eq:lstar}
\end{equation}
\end{defn}
This notion unifies several tridendriform-type structures as follows. 
\begin{remark}
	\begin{enumerate}
\item When $\kappa=0$, an \et of weight $(\lambda,\kappa)$ is called a  {\bf tridendriform algebra of weight $\lambda$}  (or simply a {\bf $\lambda$-tridendriform algebra})~\cite{BGN4,BR10,ZZWG}. In particular, when $\lambda=1$, it is just the tridendriform algebra recalled above.
\mlabel{it:rk1} 
\item
When $\lambda=0$, then the \et of weight $(\lambda,\kappa)$ is called a {\bf modified tridendriform algebra of weight $\kappa$}. Particularly, 
 when $\kappa=1$, then it recovers the modified tridendriform algebra given in \cite{Li18}. 
\mlabel{it:rk2} 
\item When $\lambda=\kappa=0$, then the \et of weight $(\lambda,\kappa)$ is also called a {\bf tridendriform algebra of weight 0}, which recovers the notion of a {\bf $\mathcal{K}$-algebra} in ~\cite{BR10,Cha02}. Note that a tridendriform algebra of weight $0$ is different from a dendriform algebra. The extended variant of the latter will be introduced in Definition~\mref{defn:ed}. 
\mlabel{it:rk3} 
\end{enumerate}
\mlabel{rk:cases}
\end{remark}

From a tridendriform algebra $(T,\pr,\su,\ob)$, there is an associative structure on $(T,\star)$ given by $\star: =\pr+\su+\ob$~\mcite{LR2}. More generally, we have 

\begin{prop}\mlabel{prop:eta} Let $(A,\pr,\su,\ob)$ be an \et of weight $(\lambda,\kappa)$. Then equipped with the multiplication $\star_\lambda$ defined in Eq.~\meqref{eq:lstar}, 
$A$ is an associative algebra.
\end{prop}

\begin{proof}
For all $x,y,z\in A$, we have
\begin{eqnarray*}
&&(x\star_\lambda y)\star_\lambda z\\
&=&(x\star_\lambda y)\pr z+(x\star_\lambda y)\su z+\lambda(x\star_\lambda y)\ob z\\
&=&(x\pr y+ x\su y+\lambda x\ob y)\pr z+ (x\pr y+ x\su y+\lambda x\ob y)\su z +\lambda(x\pr y+ x\su y+ \lambda x\ob y)\ob z\\
&=&x\pr (y\pr z+y\su z+\lambda y\ob z)+\kappa x\ob(y\ob z)+ x\su (y\pr z)
+ \lambda x\ob(y \pr z)+ x\su (y\su z)\\
&&-\kappa(x\ob y)\ob z+\lambda x\ob (y\su z)+\lambda x\su (y\ob z)+\lambda^2x\ob (y\ob z)\quad(\text{by Eqs.~\meqref{eq:e1}-\meqref{eq:e7}})\\
&=&x\pr (y\pr z+y\su z+\lambda y\ob z)+ x\su (y\pr z)
+ \lambda x\ob(y \pr z)+ x\su (y\su z)\\
&&+\lambda x\ob (y\su z)+\lambda x\su (y\ob z)+\lambda^2x\ob (y\ob z)\quad(\text{by Eq~\meqref{eq:e7}}).
\end{eqnarray*}
On the other hand, we get
\begin{eqnarray*}
	&&x\star_\lambda (y\star_\lambda z)\\
	&=&x\star_\lambda(y\pr z+y\su z+\lambda y\ob z)\\
	&=& x \pr (y\pr z+y\su z+\lambda y\ob z)+ x\su (y\pr z+y\su z+\lambda y\ob z)+\lambda x\ob (y\pr z+y\su z+\lambda y\ob z).
\end{eqnarray*}
Thus $(x\star_\lambda y)\star_\lambda z=x\star_\lambda (y\star_\lambda z)$, giving the associativity of $\star_\lambda$.
\end{proof}

The important notions of Rota-Baxter algebras and modified Rota-Baxter algebras are merged together to give \wmrbs in~\mcite{BGN2,ZGQ}.
\begin{defn}
Let $R$ be a $\bfk$-algebra and let $\lambda,\kappa\in \bfk$. A linear map $P:R\to R$ is called an {\bf \wmrbo (ERBO) of weight $(\lambda,\kappa)$} if $P$ satisfies the {\bf extended Rota-Baxter identity}
\begin{equation}P(x)P(y)=P\big(xP(y)\big)+P\big(P(x)y\big)+\lambda P(xy)+\kappa xy , \quad  x,y\in R.
\mlabel{eq:grb}
\end{equation}
Then the pair $(R,P)$ is called an {\bf\wmrb of weight $(\lambda,\kappa)$}, or simply an {\bf \wmrb}.

An \wmrb homomorphism $f: (R,P)\to (S,Q)$ between two \wmrbs $(R,P)$ and  $(S,Q)$ of the same weight $(\lambda,\kappa)$ is a $\bfk$-algebra homomorphism such that $f\circ P = Q\circ f$.
\mlabel{defn:freeGRBA}
\end{defn}

When $\kappa=0$, we recover the notion of a Rota-Baxter algebra; while when $\lambda=0$, we recover the notion of a modified Rota-Baxter algebra. 

As is well known, a Rota-Baxter algebra naturally induces a dendriform algebra or a tridendriform algebra~\mcite{Ag1,EF}. More generally, we have 

\begin{theorem}Fix $\lambda,\kappa\in\bfk$. Let $(R,P)$ be an \wmrb of weight $(\lambda,\kappa)$. Define
\begin{equation}x\prp y:=xP(y), \quad x\scp y:=P(x)y, \quad x\obp y:=  xy,\quad  x,y\in R.
\mlabel{eq:prp}
\end{equation}
Then $(R,\prp,\scp,\obp)$ is an \et of weight $(\lambda,\kappa)$.
\mlabel{thm:erbt}
\end{theorem}

In Section~\mref{sec:wmrbet} we will provide this result with an operadic interpretation which also allows us to give an inverse (Theorem~\mref{thm:owdn}), meaning that the \et is all there is that can be induced from the \wmrb. 

\begin{proof}	
We just verify the first three relations in Eqs.~\meqref{eq:e1}--\meqref{eq:e3}. The other relations are proved similarly. For all $x,y,z\in R$, we have
\begin{align*}
(x\prp y)\prp z={}&x(P(y)P(z))\\
={}&x(P(yP(z)+P(y)z+\lambda yz)+\kappa xy)\\
={}&xP(yP(z)+P(y)z+\lambda yz)+\kappa xyz\\
={}& x\prp (y\prp z+y\scp z+\lambda y\obp z)+\kappa x\obp(y\obp z).
\end{align*}
$$(x\scp y)\prp z=(P(x)y)P(z)=P(x)(yP(z))
=x\scp (y\prp z).
$$
\begin{align*}
x\scp(y\scp z)={}&P(x)(P(y)z)\\
={}&(P(xP(y)+P(x)y+\lambda xy)+\kappa xy)z\\
={}&(xP(y)+P(x)y+\lambda xy)\scp z+\kappa (xy)\obp z\\
={}&(x\prp y+x\scp y+\lambda x\obp y)\scp z+\kappa (x\obp y)\obp z.
\qedhere
\end{align*}
\end{proof}

Taking the special cases when $\lambda=0$ and furthermore $\kappa=0$, we obtain

\begin{coro} \begin{enumerate}
	\item Let $(R,P)$ be a modified Rota-Baxter algebra of weight $\kappa$. 
Then $(R,\prp,\scp,\obp)$ is a  modified tridendriform algebra of weight $\kappa$.
\mlabel{it:mo1}
\item Let $(R,P)$ be a Rota-Baxter algebra of weight $0$. 
Then $(R,\prec_P,\succ_P,\obp)$ is a tridendriform algebra of weight $0$.
\mlabel{it:mo2}
\end{enumerate}
\mlabel{co:merbt}
\end{coro}

We note again that a Rota-Baxter algebra of weight zero gives rise to a tridendriform algebra of weight zero, which is not a dendriform algebra. We next give an extended variant of the dendriform algebra. Its motivation and precise relation with the dendriform algebra will be clarified in Theorem~\mref{thm:edn}. 
In fact, an extended dendriform algebra is  a type II extended dendriform algebra of weight $(\lambda,\kappa)$ that will be defined in Definition~\mref{defn:type}.

\begin{defn}\mlabel{defn:ed}
 An {\bf \edr} is a vector space $A$ with two binary operations $\pr$ and $ \su$ that satisfy the following relations:
\begin{align}
	(x\pr y)\pr z+(x\star y)\su z&=x\pr (y\star z)+x\su (y\su z).\mlabel{eq:ed1}\\
	(x\su y)\pr z&=x\su (y\pr z),\quad x,y,z\in A. \mlabel{eq:ed2}
\end{align}
	Here  $\star =\pr +\su $.
\mlabel{defn:edri}
\end{defn}

For a given dendriform algebra, adding Eqs.~\meqref{eq:dd1} and \meqref{eq:dd3} gives Eq.~\meqref{eq:ed1}, and Eq.~\meqref{eq:dd2} is the same as Eq.~\meqref{eq:ed2}. Thus, we obtain 
\begin{prop}	\mlabel{prop:ded}
	Every dendriform algebra is an \edr.
\end{prop}
Moreover, by adding Eqs.~\meqref{eq:ed1} and ~\meqref{eq:ed2} together, we have 
\begin{prop}Let $(A,\pr,\su)$ be an \edr. Then the product $\star:=\pr+\su$ is associative.
\mlabel{prop:eda}
\end{prop}

One can also derive an \edr from an extended Rota-Baxter algebra with arbitrary weight. 

\begin{theorem}
Let $(R,P)$ be an \wmrb of weight $(\lambda,\kappa)$. Let $\prp$ and $ \scp$ be defined by 
\begin{equation*}
x\prp y:=xP(y)+\lambda xy,\quad x\scp y:=P(x)y,\quad x,y\in R.
\end{equation*}
Then $(R,\prp,\scp)$ is an \edr. 
\mlabel{thm:erbd}
\end{theorem}

\begin{proof}Let $\star_P:=\prp+\scp$. For all $x,y,z\in R$, we have 
	\begin{align*}
	(x\prp y)\prp z+(x\star_P y)\scp z={}&(xP(y)+\lambda xy)\prp z+(xP(y)+P(x)y+\lambda xy)\scp z\\
	={}&(xP(y)+\lambda xy)P(z)+\lambda xP(y)z+\lambda ^{2}xyz+P(xP(y)+P(x)y+\lambda xy)z\\
	={}&xP(y)P(z)+\lambda xyP(z)+\lambda xP(y)z+\lambda ^{2}xyz+P(x)P(y)z-\kappa xyz\\
	={}&xP(P(y)z+yP(z)+\lambda yz)+\lambda x(P(y)z+yP(z)+\lambda yz)+P(x)P(y)z\\
	={}&xP(y\star_P z)+\lambda x(y\star_P z)+P(x)(y\scp z)\\
	={}&x\prp (y\star_P z)+x\scp (y\scp z).
	\end{align*}
Furthermore,
	\begin{align*}
	(x\scp y)\prp z={}&(x\scp y)P(z)+\lambda (x\scp y)z\\
	={}&(P(x)y)P(z)+\lambda P(x)yz\\
	={}&P(x)(y\prp z)\\
	={}&x\scp (y\prp z).
	\end{align*}
Thus $(R,\prp,\scp)$ is an \edr.
\end{proof}

Taking the special case when $\lambda=0$, we obtain the following induced structure of modified Rota-Baxter algebras. 
\begin{coro}
If $(R,P)$ is a modified Rota-Baxter algebra, then $(R,\prp,\scp)$ is an \edr.
\mlabel{co:modend}
\end{coro}

\subsection{Extended tridendriform algebras and extended post-Lie algebras}
\mlabel{ss:epostlie}
This section gives the notion of an extended post-Lie algebra, which can be derived from extended tridendriform algebras; while extended post-Lie algebras give rise to Lie algebras. We show that an extended Rota-Baxter Lie algebra of weight $(\lambda,\kappa)$ also induces an extended post-Lie algebra.
\begin{defn}Let $\lambda,\kappa\in \bfk$ be fixed.
An {\bf extended post-Lie algebra of weight $(\lambda,\kappa)$} or simply an {\bf  extended post-Lie algebra} is a  vector space $A$ equipped with two
binary operations $\circ$ and $[\,,\,]$, where $[\,,\,]$ is a Lie bracket and $\circ$ satisfies the  following  compatibility conditions: for all $x,y,z\in A$,
\begin{align}
(x\circ y)\circ z -(y\circ x)\circ z +\lambda [x, y]\circ z +\kappa [[x,y],z]&= x\circ (y\circ z)-y\circ (x\circ z),\mlabel{eq:post1}\\
x  \circ [y,z] &= [x\circ y , z ] +[y,  x\circ z]. \mlabel{eq:post2}
\end{align}
\end{defn}
\begin{remark}
\begin{enumerate}
\item
 When $\kappa=0$, then an extended post-Lie algebra $(A,[\,,\,],\circ)$ of weight $(\lambda,\kappa)$ is
called a {\bf post-Lie algebra of weight $\lambda$}, that is, $[\,,\,]$ is a Lie bracket and $\circ$ satisfies Eq.~\meqref{eq:post2} and 
	\begin{equation}
		(x\circ y)\circ z -(y\circ x)\circ z +\lambda [x, y]\circ z=x\circ (y\circ z)-y\circ (x\circ z),\quad x,y,z\in A.\mlabel{eq:qpost1}
	\end{equation}
In particular, if $\lambda=1$, then a post-Lie algebra of weight $\lambda$ is a classical post-Lie algebra.
\mlabel{it:epost1}
\item\mlabel{it:epost2} When $\lambda=0$, then an extended post-Lie algebra $(A,[\, ,\,],\circ)$ of weight $(\lambda,\kappa)$ is called a {\bf modified post-Lie algebra of weight $\kappa$}, that is, $[\,,\,]$ is a Lie bracket and $\circ$ satisfies Eq.~\meqref{eq:post2} and
	\begin{equation}
		(x\circ y)\circ z -(y\circ x)\circ z + \kappa [[x,y],z]=x\circ (y\circ z)-y\circ (x\circ z),\quad x,y,z\in A.\mlabel{eq:mpost1}
	\end{equation}
\item
When $\lambda=\kappa=0$,  an extended post-Lie algebra $(A,[\,,\,],\circ)$ of weight $(\lambda,\kappa)$ is called a {\bf post-Lie algebra of weight $0$}. 
\mlabel{it:epost3}
\item Let $(A, \circ,[\,,\,])$ be an extended post-Lie algebra. If $[\,,\,]$ is abelian, then  $(A, \circ)$ is a  pre-Lie algebra. 
\end{enumerate}
\mlabel{rk:epost}
\end{remark}

\begin{prop}
Let $(A,\pr,\su,\ob)$ be an \et of weight $(\lambda,\kappa)$. Define 
	\begin{align*}
		[x,y]:&=x\ob y-y\ob x, \\
		x\circ y:&= x\succ y -y\prec x,\quad x,y\in A. 
	\end{align*}
	Then $(A,[\ ,\ ],\circ)$ is an extended post-Lie algebra of weight $(\lambda,\kappa)$.
\mlabel{prop:etap}
\end{prop}
\begin{proof}
By Definition~\mref{defn:etri},	$(A,\ob)$ is an associative algebra,  and so $[\, ,\,]$ is a Lie bracket. We shall now verify Eqs.~\meqref{eq:post1} and ~\meqref{eq:post2}. For all $x,y,z\in A$, we have
\begin{align*}
	&(x\circ y)\circ z - (y\circ x)\circ z+\lambda [x, y]\circ z +\kappa [[x,y],z]- x\circ (y\circ z)+y\circ (x\circ z)\\
	=&\big((x\circ y)\succ z-z\prec (x\circ y)\big)-\big((y\circ x)\succ z-z\prec (y\circ x)\big)+\lambda \big([x,y]\succ z-z\prec [x,y]\big)\\
	&+\kappa \big([x,y]\ob z-z\ob [x,y]\big)-\big(x\succ (y\circ z)-(y\circ z)\prec x\big)+\big(y\succ(x\circ z)- (x\circ z)\prec y\big)\\
	=&\big((x\succ y -y\prec x)\succ z-z\prec (x\succ y -y\prec x)\big)-\big((y\succ x-x\prec y)\succ z-z\prec (y\succ x-x\prec y)\big)\\
	&+\lambda \big((x\ob y-y\ob x)\succ z -z\prec (x\ob y-y\ob x)\big)+\kappa \big((x\ob y-y\ob x)\ob z-z\ob (x\ob y-y\ob x)\big)\\
	&-\big(x\succ (y\succ z-z\prec y)-(y\succ z-z\prec y)\prec x\big)+\big(y\succ (x\succ z-z\prec x)-  (x\succ z-z\prec x)\prec y\big)\\
	=&\Big((x\succ y+x\prec y+\lambda x\ob y)\succ z+\kappa (x\ob y)\ob z-x\succ (y\succ z)\Big)\\
	&-\Big((y\prec x+y\succ x+\lambda y\ob x)\succ z+\kappa (y\ob x)\ob z-y\succ(x\succ z)\Big)\\
	&+\Big((z\prec x)\prec y-z\prec (x\prec y+x\succ y+\lambda x\ob y)-\kappa z\ob (x\ob y)\Big)\\
	&+\Big(z\prec (y\prec x+y\succ x+\lambda y\ob x)+\kappa z\ob (y\ob x)-(z\prec y)\prec x\Big)\\
	&+\Big(x\succ (z\prec y)-(x\succ z)\prec y\Big)+\Big((y\succ z)\prec x-y\succ (z\prec x)\Big)\\
	=&0
\end{align*}
by Eqs.~\meqref{eq:e1} and ~\meqref{eq:e3}. Then we obtain Eq.~\meqref{eq:post1}.
On the other hand, we have
		\begin{align*}
			&x\circ [y,z]-[x\circ y,z]-[y,x\circ z]
			\\=&(x\succ [y,z]-[y,z]\prec x)-((x\circ y)\ob z-z\ob (x\circ y))-(y\ob (x\circ z)-(x\circ z)\ob y)
			\\=&(x\succ (y\ob z-z\ob y)-(y\ob z-z\ob y)\prec x)
			\\&-((x\succ y-y\prec x)\ob z-z\ob(x\succ y-y\prec x))
			\\&-(y\ob(x\succ z-z\prec x)-(x\succ z-z\prec x)\ob y)
			\\=&(x\succ (y\ob z)-(x\succ y)\ob z)-(x\succ (z\ob y)-(x\succ z)\ob y)
			\\&-((y\ob z)\prec x-y\ob(z\prec x))+((z\ob y)\prec x-z\ob(y\prec x))
			\\&+((y\prec x)\ob z-y\ob(x\succ z))+(z\ob(x\succ y)-(z\prec x)\ob y)\\
			=&0
		\end{align*}
by  Eqs.~\meqref{eq:e4} and ~\meqref{eq:e7}. Thus, Eq.~\meqref{eq:post2} holds, and so $(A,[\ ,\ ],\circ)$ is an extended post-Lie algebra.
\end{proof}

Taking $\lambda=0$ in Proposition~\mref{prop:etap}, we obtain
\begin{coro}
Let $(A,\pr,\su,\ob)$ be a modified tridendriform algebra of weight $\kappa$. Define 
	\begin{align*}
		[x,y]:&=x\ob y-y\ob x, \\
		x\circ y:&= x\succ y -y\prec x,\quad x,y\in A. 
	\end{align*}
	Then $(A,[\ ,\ ],\circ)$ is a modified post-Lie algebra of weight $\kappa$.
	\mlabel{prop:mtmp}
\end{coro}
Taking $\kappa=0$ in Proposition~\mref{prop:etap}, we obtain
\begin{coro}
	Let $(A,\pr,\su,\ob)$ be a tridendriform algebra of weight $\lambda$. Define 
	\begin{align*}
		[x,y]:&=x\ob y-y\ob x, \\
		x\circ y:&= x\succ y -y\prec x,\quad x,y\in A. 
	\end{align*}
	Then $(A,[\ ,\ ],\circ)$ is a post-Lie algebra of weight $\lambda$. 
	\mlabel{prop:triap}
\end{coro}

Generalizing the classical relation between post-Lie algebras and Lie algebras~\cite{Va07}, we obtain

\begin{prop}
Let $(A, \circ,[\,,\,])$ be an extended post-Lie algebra. Define a binary operation on $A$ by 
$$\{x,y\}:=x\circ y -y \circ x +\lambda [x,y],\quad x,y\in A.$$
Then $\{\, ,\,\}$ is a Lie bracket. 
\mlabel{prop:epoll}
\end{prop}

\begin{proof}
Let $x,y,z\in A$. First, $\{\, ,\,\}$ is skew-symmetric, since
		\begin{align*}
			\{x,y\}&=x\circ y -y \circ x +\lambda [x,y]
			\\&=-(y\circ x-x\circ y+\lambda [y,x])\quad(\text{by the skew-symmetry of Lie bracket $[,]$})
			\\&=-\{y,x\}.&
		\end{align*}
	
We next verify that $\{\, ,\,\}$ satisfies the Jacobi identity. Indeed, we have
	\begin{align*}
			&\{\{x,y\},z\}+\{\{y,z\},x\}+\{\{z,x\},y\}\\
			=&\big(\{x,y\}\circ z-z\circ \{x,y\}+\lambda [\{x,y\},z]\big)+\big(\{y,z\}\circ x-x\circ \{y,z\}+\lambda [\{y,z\},x]\big)\\
&+\big(\{z,x\}\circ y-y\circ \{z,x\}+\lambda [\{z,x\},y]\big)&\\
=&(x\circ y-y\circ x+\lambda [x,y])\circ z-z\circ (x\circ y-y\circ x+\lambda [x,y])+\lambda [x\circ y-y\circ x+\lambda [x,y],z]
			\\&+(y\circ z-z\circ y+\lambda [y,z])\circ x-x\circ (y\circ z-z\circ y+\lambda [y,z])+\lambda [y\circ z-z\circ y+\lambda [y,z],x]
			\\&+(z\circ x-x\circ z+\lambda [z,x])\circ y-y\circ (z\circ x-x\circ z+\lambda [z,x])+\lambda [z\circ x-x\circ z+\lambda [z,x],y]
			\\=&\bigg((x\circ y-y\circ x+\lambda [x,y])\circ z-x\circ (y\circ z)+y \circ (x\circ z)\bigg)
			\\&+\bigg((y\circ z-z\circ y+\lambda [y,z])\circ x+z\circ (y\circ x)-y\circ (z\circ x)\bigg)
			\\&+\bigg((z\circ x-x\circ z+\lambda [z,x])\circ y-z\circ (x\circ y)+x\circ (z\circ y)\bigg)\\
			&-\lambda\bigg( z\circ [x,y]+[z\circ y,x]-[z\circ x,y]\bigg)
			+\lambda \bigg([x\circ y,z]-x\circ[y,z]-[x\circ z,y]\bigg)\\
			&-\lambda \bigg([y\circ x,z]-[y\circ z,x]+y\circ [z,x]\bigg)
			+\lambda^2\bigg([[x,y],z]+[[y,z],x]+[[z,x],y]\bigg) \\
=&-\kappa \Big([[x,y],z]+ [[y,z],x]+ [[z,x],y]\Big)+\lambda^2\Big([[x,y],z]+[[y,z],x]+[[z,x],y]\Big)\\
&\quad(\text{by  Eqs.~\meqref{eq:post1} and ~\meqref{eq:post2} })\\
=&0\quad(\text{by the Jacobi identity for $[\, ,\,]$}),
\end{align*}
as desired.		
\end{proof}

We now consider extended Rota-Baxter operators on Lie algebras. 
\begin{defn}
Let $\lambda,\kappa\in \bfk$ be fixed.
An {\bf extended Rota-Baxter Lie  algebra of weight $(\lambda,\kappa)$} or simply an {\bf  extended Rota-Baxter Lie algebra} is a  vector space $\frakg$ together with a Lie bracket $[\,,\,]$ and a linear operator $P$ such that
\begin{equation}
[P(x),P(y)]=P([P(x),y]+[x,P(y)]+\lambda[x,y])+\kappa[x,y],\quad x,y\in\frakg.
\mlabel{eq:erbal}
\end{equation}
\end{defn}
Similar to ~\cite[Corollary~5.6]{BGN}, we have
	\begin{theorem}
		Let $\frakg:=(\frakg,[\ ,\ ],P)$ be an extended Rota-Baxter Lie algebra of weight $(\lambda,\kappa)$. Then the binary operation
\begin{equation*}
x\circ y:=[P(x),y],\quad x,y\in \frakg, 
\end{equation*}
makes $(\frakg, [ \ , \ ],\circ)$ into an extended post-Lie algebra of weight $(\lambda,\kappa)$.
		\mlabel{prop:eepost}
	\end{theorem}
	\begin{proof}
For all $x,y,z\in\frakg$, we have
\begin{align*}
&(x\circ y)\circ z - (y\circ x)\circ z+\lambda [x, y]\circ z +\kappa [[x,y],z]\\
=&[P(x\circ y),z]-[P(y\circ x),z]+\lambda [P([x,y]),z]+\kappa [[x,y],z]\\
=&\big[P([P(x),y]),z\big]-\big[P([P(y),x]),z\big]+\lambda [P[x,y],z]+\kappa [[x,y],z]\\
=&\big[P([P(x),y]+[x,P(y)]+\lambda[x,y])+\kappa[x,y],z\big]\\
=&\big[[P(x),P(y)],z\big]   \qquad(\text{by Eq.~\meqref{eq:erbal}})\\
=&\big[P(x),[P(y),z]\big]+\big[P(y),[z,P(x)]\big] \qquad(\text{by the Jacobi identity})\\
=&x\circ(y\circ z)-y\circ (x\circ z).
\end{align*}
\begin{align*}
x\circ[y,z]=&[P(x),[y,z]]\\
=&[[P(x),y],z]+[y,[P(x),z]]\qquad(\text{by the Jacobi identity})\\
=&[x\circ y,z]+[y,x\circ z].
\end{align*}
Thus,  Eqs.~\meqref{eq:post1} and ~\meqref{eq:post2} hold.
\end{proof}

Taking the special cases  when $\lambda=0$ or $\kappa=0$ or both, we obtain

\begin{coro} With the notations as above, we have \begin{enumerate}
	\item If $(\frakg, [ \ , \ ], P)$ is a modified Rota-Baxter Lie algebra of weight $\kappa$, then
 $(\frakg, [ \ , \ ],\circ)$ is a modified post-Lie algebra of weight $\kappa$.
\mlabel{it:mp1}
	\item If $(\frakg, [ \ , \ ], P)$ is a Rota-Baxter Lie algebra of weight $\lambda$, 
then  $(\frakg, [ \ , \ ],\circ)$ is a post-Lie algebra of weight $\lambda$.
\mlabel{it:mp2}
\item If $(\frakg, [ \ , \ ], P)$ is a Rota-Baxter Lie algebra of  weight $0$, then $(\frakg, [ \ , \ ],\circ)$ is a post-Lie algebra of weight $0$.
\mlabel{it:mp3}
\end{enumerate}
\mlabel{co:postl}
\end{coro}
We give an example of extended Rota-Baxter operators and post-Lie algebras on a Lie algebra. 
\begin{exam}
Let $\frakg=\mathrm{sl(2,\mathbb{C})}$ be the $3$-dimensional special linear
 Lie algebra over $\CC$. 
Let
$$
e = \left(\begin{array}{cc} 0&1\\0&0 \end{array}\right), \quad f = \left(\begin{array}{cc} 0&0\\1&0 \end{array}\right),\quad h = \left(\begin{array}{cc} 1&0\\0&-1 \end{array}\right)
$$
be the Cartan-Weyl basis of $\speciall$. Then
\begin{equation}\label{eq:product}
[h,e] = 2e, \quad [h,f] = -2f, \quad [e,f] = h.
\end{equation}
Thus, a linear operator $P:\frakg\to \frakg$ is determined by
\begin{equation*}
\left( 
P(e),\;
P(f),\; 
P(h) \right)  =\left( 
e,\;
f,\;
h \right) \left( \begin{array} {ccc}
r_{11}&r_{12}&r_{13}\\
r_{21}&r_{22}&r_{23}\\
r_{31}&r_{32}&r_{33} \end{array}\right),
\end{equation*}
where $r_{ij}\in \CC, 1\leq i, j\leq 3$. So we  write $P$ often as  its matrix  relative to the given basis. The linear operator
$P$ is an extended Rota-Baxter operator of weight $(\lambda,\kappa)$ on $\frakg$ if the above matrix
$(r_{ij})_{3 \times 3}$ satisfies Eq.~(\ref{eq:erbal}) for $x, y \in\{e, f, h\}$. 
Let
$$P=\begin{pmatrix}
			-2&0&-\dfrac{3}{2}\\
			0&1&-2\\
			1&\dfrac{3}{4}&-\dfrac{1}{2}
\end{pmatrix}.
$$
Then by a direct computation, $P$ is an extended Rota-Baxter operator of weight $(1,1)$ on $\frakg$. Moreover,  by Proposition~\mref{prop:eepost}, there exists an extended post-Lie algebra structure  $(\frakg,[\ ,\ ] ,\circ )$  on $\frakg$, where
the binary product $\circ$ is given by the following Cayley table.

\begin{center}\begin{tabular}{c | c c c }
$\circ$	& e & f & h \\
\hline 	
	$e$ & $2e$ & $-2f-2h$ & $4e$  \\
	$f$ & $\frac{3}{2}e-h$ & $-\frac{3}{2}f$ & $2f$ \\
	$h$ & $-2e+2h$ & $f-\frac{3}{2}h$ & $3e-4f$  \\
\end{tabular}
\end{center}
\end{exam}

\subsection{Extended dendriform algebras and extended pre-Lie algebras}
\mlabel{ss:eprelie}
Inspired by the classical relation between dendriform algebras and pre-Lie algebras, we give the notion of an extended pre-Lie algebra. Then we establish a commutative diagram in the framework of extended dendriform algebras and extended pre-Lie algebras, generalizing the one given in Eq.~\meqref{eq:dcom}. 
\begin{defn}An {\bf extended pre-Lie algebra} is a vector space $A$ with three binary operations $\tril,\trir,\circ:A\ot A\to A$ such that for all $x,y,z\in A$,
\begin{equation}
(x\tril y)\tril z-x\tril(y\ast z)+(x\ast y)\trir z-x\trir(y\trir z)
= (y\tril x)\tril z-y\tril(x\ast z)+(y\ast x)\trir z-y\trir(x\trir z).\mlabel{eq:epre1}
\end{equation}
\begin{equation}
\begin{split}
	&(x\circ y)\circ z-x\circ (y\circ z)+(x\tril y)\tril z-x\tril (y\ast z)+z\trir(y\trir x)-(z\ast y)\trir x\\
=&(y\circ x)\circ z-y\circ (x\circ z)+(y\tril x)\tril z-y\tril (x\ast z)+z\trir(x\trir y)-(z\ast x)\trir y.
\end{split}
\mlabel{eq:epre2}
\end{equation}
Here $\ast:=\tril+\trir$.
\end{defn}
\begin{remark}An extended pre-Lie algebra $(A,\tril,\trir,\circ)$ with $\tril=\trir=0$ is just a pre-Lie algebra.	
	From this perspective, an extended pre-Lie algebra is a generalization of pre-Lie algebras.
\end{remark}

\begin{prop}Let $(A,\prec,\succ)$ be an \edr. Define 
$$x\tril y:=x\prec y,\quad x \trir y:=x\succ y,\quad x\circ y:=x\succ y-y\prec x,\quad x,y\in A.$$
Then $(A,\tril,\trir, \circ)$ is an extended pre-Lie algebra.
\mlabel{prop:edrpre}
\end{prop}
\begin{proof}
By Eq.~\meqref{eq:ed1},  we obtain Eq.~\meqref{eq:epre1}.	
Let $\ast:=\tril+\trir$. For all $x,y,z\in A$, we have 
\begin{eqnarray*}
&&(x\circ y)\circ z-x\circ (y\circ z)-(y\circ x)\circ z+y\circ (x\circ z)\\
&=&(x\succ y-y\prec x) \succ z-z\prec (x\succ y-y\prec x)-x\succ (y\succ z-z\prec y)+(y\succ z-z\prec y)\prec x\\
&&-(y\succ x-x\prec y) \succ z+z\prec (y\succ x-x\prec y)+y\succ (x\succ z-z\prec x)-(x\succ z-z\prec x)\prec y\\
&=&(x\succ y-y\prec x) \succ z-z\prec (x\succ y-y\prec x)-x\succ (y\succ z)-(z\prec y)\prec x\\
&&-(y\succ x-x\prec y) \succ z+z\prec (y\succ x-x\prec y)+y\succ (x\succ z)+(z\prec x)\prec y\quad(\text{by Eq.~\meqref{eq:ed2}})\\
&=&y\succ (x\succ z)-(y\prec x+y\succ x)\succ z+(x\prec y+x\succ y) \succ z-x\succ (y\succ z)\\
&&+(z\prec x)\prec y-z\prec (x\succ y+x\prec y)+z\prec (y\succ x+y\prec x)-(z\prec y)\prec x\\
&&\quad(\text{by rearranging the order of elements})\\
&=&(y\prec x)\prec z-y\prec(x\ast z)+x\prec(y\ast z)-(x\prec y)\prec z\\
&&+z\succ (x\succ y)-(z\ast x)\succ y+(z\ast y)\succ x- z\succ(y\succ x)\quad(\text{by Eq.~\meqref{eq:ed1}})\\
&=&(y\tril x)\tril z-y\tril(x\ast z)+x\tril(y\ast z)-(x\tril y)\tril z\\
&&+z\trir (x\trir y)-(z\ast x)\trir y+(z\ast y)\trir x- z\trir(y\trir x),
\end{eqnarray*}
proving Eq.~\meqref{eq:epre2}.
\end{proof}

\begin{prop}Let $(A,\tril,\trir,\circ)$ be an extended pre-Lie algebra. Define 
	$$ [x,\,y]:=x\circ y-y\circ x,\quad x,y\in A.$$
Then $(A,[\,,\,])$ is a Lie algebra.
\mlabel{prop:eprlie}
\end{prop}
\begin{proof}
Let $\ast:=\tril+\trir$. For all $x,y,z\in A$, we have 
\begin{eqnarray*}
&&[[x,y],z]+[[y,z],x]+[[z,x],y]\\
&=&(x\circ y-y\circ x)\circ z-z\circ(x\circ y-y\circ x)+(y\circ z-z\circ y)\circ x\\
&&-x\circ(y\circ z-z\circ y)+(z\circ x-x\circ z)\circ y-y\circ(z\circ x-x\circ z)\\
&=&\underline{(y\tril x)\tril z}-\underline{y\tril (x\ast z)}+z\trir(x\trir y)-(z\ast x)\trir y\\
&&\underline{-(x\tril y)\tril z}+\underline{x\tril (y\ast z)}-z\trir(y\trir x)+(z\ast y)\trir x\\
&&+(z\tril y)\tril x-z\tril (y\ast x)+\underline{x\trir(y\trir z)}-\underline{(x\ast y)\trir z}\\
&&-(y\tril z)\tril x+y\tril (z\ast x)-x\trir(z\trir y)+(x\ast z)\trir y\\
&&+(x\tril z)\tril y-x\tril (z\ast y)+y\trir(z\trir x)-(y\ast z)\trir x\\
&&-(z\tril x)\tril y+z\tril (x\ast y)-\underline{y\trir(x\trir z)}+\underline{(y\ast x)\trir z}\quad(\text{by Eq.~\meqref{eq:epre2}})\\
&=&z\trir(x\trir y)-(z\ast x)\trir y-{z\trir(y\trir x)}+{(z\ast y)\trir x}\\
&&+{(z\tril y)\tril x}-{z\tril (y\ast x)}-{(y\tril z)\tril x}+{y\tril (z\ast x)}-x\trir(z\trir y)+(x\ast z)\trir y\\
&&+(x\tril z)\tril y-x\tril (z\ast y)+y\trir(z\trir x)-{(y\ast z)\trir x}-{(z\tril x)\tril y}+z\tril (x\ast y)\\
&&(\text{by applying Eq.~\meqref{eq:epre1} to the above underlined terms})\\
&=&z\trir(x\trir y)-(z\ast x)\trir y-x\trir(z\trir y)+(x\ast z)\trir y\\
&&+(x\tril z)\tril y-x\tril (z\ast y)-(z\tril x)\tril y+z\tril (x\ast y)\quad(\text{by  Eq.~\meqref{eq:epre1}})\\
&=&0\quad(\text{by Eq.~\meqref{eq:epre1}}),
\end{eqnarray*}
giving the Jacobi identity. Thus $(A,[\,,\,])$ is a Lie algebra.
\end{proof}

In summary, putting Propositions~\mref{prop:eta}, \mref{prop:eda}, \mref{prop:etap}, \mref{prop:epoll}, \mref{prop:edrpre} and  \mref{prop:eprlie}  together, we obtain 

\begin{theorem} \mlabel{thm:edcom}
These extended structures fit into the following commutative diagram. 
\begin{equation}
{\small \begin{split}
	\xymatrix{\txt{\rm  Extended tridendriform\\ \rm algebra}\ar[d]_{[x,y]:=x\ob y-y\ob x}^{x\circ y:= x\succ y -y\prec x } \ar@/2pc/[rr]^--{\quad x\star y:=x\pr y+x\su y +\lambda x\ob y}\ar[d]&&\txt{\rm\quad Associative\\ \rm algebra}\ar[d]_{[x,y]'=x\star y-y\star x} \ar[d]&& \txt{\rm  Extended dendriform \\ \rm algebra} \ar@/2pc/[ll]_--{ x\star y:=x\su' y+x\pr' y}\ar[d]_{x \tril y:=x\prec' y,\,x\trir y:=x\succ' y}^{x\circ y:=x\su' y-y\pr' x}&\\
		\txt{\rm Extended post-Lie\\ \rm algebra}\ar[rr]_--{ \{x,y\}:=x\circ y -y \circ x +\lambda[x,y]}&&\txt{\rm  Lie \\ \rm algebra} 
		&& \txt{\rm Extended pre-Lie\\ \rm algebra}\ar[ll]^--{\quad\; [x,\,y]:=x\circ y-y\circ x}&
	}
\mlabel{eq:edcom2}
\end{split}
}
\end{equation}
\end{theorem}

\section{The free  \wmrb on an algebra}
\mlabel{sec:FERBA}
In this section,  an explicit construction of  the free extended Rota-Baxter algebra on a nonunitary  algebra is constructed using bracketed words, based on which the free extended Rota-Baxter algebra on a vector space is also obtained. Other than its own interest, this construction provides a key ingredient in the operadic study of extended structures in Section~\mref{sec:wmrbet}.

\subsection{Free operated semigroups and free operated algebras}
For later applications, we recall the construction of free operated semigroups and free operated algebras from~\mcite{Guo09}. 

\begin{defn}
		{\rm An {\bf operated semigroup} (resp. {\bf operated $\bfk$-algebra}) is a semigroup (resp. $\bfk$-algebra) $U$ together with a map (resp. $\bfk$-linear map) $P: U\to U$. A homomorphism from an operated semigroup \, (resp. $\bfk$-algebra) $(U,P_U)$ to an operated semigroup (resp. $\bfk$-algebra) $(V,P_V)$ is a semigroup (resp. $\bfk$-algebra) homomorphism $f :U\to V$ such that $f \circ P_U= P_V \circ f$. } \label{de:mapset}
\end{defn}

For any set $Y$, let $S(Y)$ be the free semigroup on $Y$. Let $\lc Y\rc:=\{\lc y\rc \,|\, y\in Y\}$ denote a replica of $Y$, which is disjoint from $Y$.

Let $X$ be a given set. The construction of  the free operated semigroup on  $X$ can be obtained from a direct system of free semigroups as follows. First let $\fs_0:=S(X)$, and let
$$ \fs_1:=S(X\cup \lc \fs_0\rc).$$
Let $\iota_{0}$ be the natural embedding $\iota_{0}:\fs_0 \hookrightarrow   \fs_1$.
For a given $n\geq 1$, assume that the free semigroups $\fs_i$  for $0 \leq i\leq n$ have been defined, with the properties that for $0 \leq i\leq n-1$ there are the equalities $\fs_{i+1}=S(X\cup \lc\fs_{i}\rc )$ and natural embeddings  $ \iota_i: \fs_i \to \fs_{i+1}$. Then let
\begin{equation}
	\fs_{n+1}:=S(X\cup \lc\fs_{n}\rc ).
	\label{eq:frakm}
\end{equation}
Now the embedding $\iota_{n-1}$ induces an injection from $X\cup \lc\fs_{n-1}\rc$ to $X\cup \lc \fs_{n} \rc$.
The functoriality of free semigroups gives an embedding
\begin{equation}
	\iota_{n}: \fs_{n} = S(X\cup \lc\fs_{n-1}\rc)\hookrightarrow
	S(X\cup \lc \fs_{n}\rc) = \fs_{n+1}. \label{eq:tranm}
\end{equation}
Thus we obtain a  direct system
$\{\,\iota_{n}: \fs_n\to \fs_{n+1}\, \}_{n=0}^\infty.$
Denote
$$\fs(X):=\bigcup_{n\geq 0}\fs_n=\dirlim\fs_n.$$
Note that $\fs(X)$ is a semigroup. An element in $\fs(X)$ is called a {\bf bracketed word}.

Let $\bfk\fs(X)$ be the vector space with basis
$\fs(X)$, equipped with the semigroup algebra structure from the semigroup $\fs(X)$.
Further, the map
$$\lc\ \rc: \fs(X) \to \fs(X), w \mapsto \lc w\rc$$
is extended by $\bfk$-linearity to a linear operator $P_{\lc\ \rc}$ on $\bfk\fs(X)$. Then $(\bfk\fs(X),P_{\lc\ \rc})$ is an operated
$\bfk$-algebra.

\begin{lemma}\cite{Guo09}
	Let $i_X:X \to \fs(X)$ and $j_X: \fs(X) \to \bfk\fs(X)$ be the natural embeddings. Then
	\begin{enumerate}
		\item
		The triple $(\fs(X),\lc\ \rc, i_X)$ is the free operated semigroup on $X$.
		\label{it:mapsetm}
		\item
		The triple $(\bfk\fs(X),P_{\lc\ \rc}, j_X\circ i_X)$ is the free operated $\bfk$-algebra
		on $X$. \label{it:mapalgsg}
	\end{enumerate}
	\label{pp:freetm}
\end{lemma}
\begin{defn}
	Let $Y,Z$ be two subsets of $\fs(X)$. Define the {\bf alternating product} of $Y$ and $Z$ to be
	\begin{equation*}
		\varLambda(Y,Z)=\Big(\bigcup_{r\geq 1}(Y\lf Z\rf)^r\Big)\bigcup\Big(\bigcup_{r\geq 0}(Y\lf Z\rf)^rY\Big)\bigcup\Big(\bigcup_{r\geq 1}(\lf Z\rf Y)^r\Big)\bigcup\Big(\bigcup_{r\geq 0}(\lf Z\rf Y)^r\lf Z\rf\Big).
	\end{equation*}
\end{defn}	
We next construct a sequence $\frakX_n$ of subsets of $\fs(X)$ by the following recursion on $n\geq 0$. For the initial step, we define $\frakX_0:=S(X)$. For the inductive step, we define
$$ \frakX_{n+1}:=\varLambda(S(X),\frakX_n), n\geq 0.$$
Then $\frakX_{0}\subseteq\frakX_{1}$. Assume $\frakX_{n-1}\subseteq\frakX_{n}$. Then $\lf \frakX_{n-1}\rf \subseteq \lf \frakX_n\rf$. By the definition of $\varLambda$, we have
$$\frakX_{n}=\varLambda(X,\frakX_{n-1})\subseteq\varLambda(X,\frakX_{n})=\frakX_{n+1}.$$

Let
\begin{equation} \frakX_{\infty}:=\lim_{\longrightarrow}\frakX_{n}=\bigcup_{n\geq 0}\frakX_{n}.
	\mlabel{eq:xinf}
\end{equation}
Elements of $\frakX_\infty$ are called the {\bf Rota-Baxter bracketed words (RBWs)}, since they form the basis of free Rota-Baxter algebras on a set~\cite{EG1,Gub}.

Recall~\cite[Lemma~3.3]{ZGG16} that every RBW $w$ has a unique {\bf alternating decomposition}:
\begin{equation}\mlabel{eq:alt}
	w=w_1 \cdots w_m,
\end{equation}
where $w_i\in X\cup \lf \frakX_\infty \rf$ for $1\le i\le m$ with $ m\geq 1$, and  no consecutive elements in the sequence
$w_1,\cdots,w_m$ are in $\lf \frakX_\infty \rf$. If $w_1$ is in $X$ (resp. $\lf \frakX_\infty \rf$),  the {\bf head} of $w$, denoted by  $h(w)$, is defined to be $0$ (resp. $1$). Similarly, if  $w_m$ is in  $X$ (resp. $\lf \frakX_\infty \rf$),  the {\bf tail} of $w$, denoted by $t(w)$, is defined to be $0$ (resp. $1$).
The natural number $m$ is called the {\bf breadth} of $w$, denoted by $\bre(w)$, with the convention that $\bre(1)=0$. We define the {\bf depth} of $w$ to be
$$\dep(w):=\min\{n\,|\,w\in \frakX_{n}\}.$$

\subsection{Construction of free  \wmrbs}
\mlabel{subsec:frerba}

The class of \wmrbs of weight $(\lambda,\kappa)$ forms a category with the \wmrb homomorphisms.
We first give the definition of the free \wmrb on a $\bfk$-algebra.
\begin{defn}
	Let $A$ be a $\bfk$-algebra. A {\bf free \wmrb} of weight $(\lambda, \kappa)$ on $A$ is an \wmrb $(F(A),P_A)$ together with an algebra homomorphism $j_A: A\to F(A)$ with the universal property that, for any  \wmrb $(R,P)$ of weight $(\lambda,\kappa)$ and any algebra homomorphism $f: A\to R$, there exists a unique \wmrb homomorphism $\tilde{f}:F(A)\to R$  such that $f=\free{f}\circ j_A$.
\end{defn}	

Fix a linear basis $X$ of $A$ as a vector space on $\bfk$. 
Let
$$\ce:=\bfk \frakX_{\infty}$$
be the vector space with basis $\frakX_\infty$ defined in Eq.~\meqref{eq:xinf}.
We now define a product $\dg$ on $\ce$.  It suffices to define $u\dg v$ for all basis elements $u,v\in\frakX_\infty$.  We proceed by induction on the sum $n:=\dep(u)+\dep(v)\geq 0$ of the depths of $u$ and $v$.  For $n=0$, we have $\dep(u)=\dep(v)=0$, and so $u,v$ are in $\frakX_{0} ~(=S(X))$. Then we define $u\dg v=uv$, the concatenation.
For a given $k\geq 0$, assume that $u \dg v$ has been defined for all $u,v\in \frakX_{\infty}$ with $0\le n\le k$. Consider $u, v\in \frakX_\infty$ with $d(u)+d(v)=k+1$. There are two cases, depending on whether or not $\bre(u)=\bre(v)=1$. 

\noindent 
{\bf Case 1. } 
If $\bre(u)=\bre(v)=1$, that is, $u$ and $v$ are in $X$ or in $\lf \frakX_{\infty} \rf$, then
\begin{equation}\mlabel{eq:de}
	u \dg v = \left \{\begin{array}{ll}
		\lf\bar{u} \dg \lf\bar{v}\rf\rf+\lf \lf\bar{u}\rf\dg \bar{v}\rf+\lambda\lf\bar{u} \dg \bar{v}\rf+\kappa\bar{u} \dg \bar{v},
		& \text{if} \ u=\lf\bar{u}\rf \  \text{and} \ v=\lf\bar{v}\rf,\\
		u v, & \text{otherwise}.
	\end{array} \right .
\end{equation}
Here the product in the first line is well defined by the induction hypothesis, since
$$\dep(\bar{u})+\dep(\lf\bar{v}\rf)=\dep(\lf\bar{u}\rf)+\dep(\bar{v})=k\quad\text{and}\quad \dep(\bar{u})+\dep(\bar{v})=k-1.$$
The product in the second line is the concatenation.

\noindent
{\bf Case 2. } 
If  $\bre(u)>1$ or $\bre(v)>1$, then by Eq.~\meqref{eq:alt},  we write
$u=u_1\cdots u_m$ and $v=v_1\cdots v_{n}$.  Define
\begin{equation}\mlabel{eq:dd}
	u \dg v=u_1\cdots u_{m-1}(u_m\dg v_1)v_2\cdots v_{n},
\end{equation}
where $w_m \dg v_1$ has been  defined by Eq.~(\mref{eq:de}) and the remaining products are the concatenation. Extending $\dg$ by bilinearity, we obtain  a binary operation $\dg$ on $\ce$.
For all $u \in \frakX_{\infty}$,  we have $\lf u\rf\in \frakX_\infty$. Thus, the restriction of  $P_{\lf\ \rf}$ on $\bfk\fs(X)$ to $\ce$ gives a linear operator
$$\pe:\ce\to\ce,\   w\mapsto\lf w\rf.$$
In the following, we will use the  notation $\lf\,\rf$ interchangeably with  $\pe$.
Let
$$j_A: A\hookrightarrow \ce $$
to be the natural injection.

The main result of this section is given as follows, and its proof will be presented in the next subsection.

\begin{theorem}
Let $A$ be an algebra. 
\begin{enumerate}
\item
The pair $(\ce,\dg)$ is an algebra.
\mlabel{it:alg}
\item
The triple $(\ce,\dg,\pe)$ is an \wmrb of weight $(\lambda,\kappa)$.
\mlabel{it:erb}
\item
The quadruple $(\ce,\dg,\pe,j_A)$ is the free \wmrb of weight $(\lambda,\kappa)$ on $A$.
\mlabel{it:ferb}
\end{enumerate}
\mlabel{thm:free}
\end{theorem}
The following corollary will be used in Section~\mref{sec:wmrbet}. 

\begin{coro}Let $M$ be a vector space over $\bfk$. Let 
$$T^+(M):=\bigoplus_{n\geq 1} M^{\ot n}$$
be the reduced tensor algebra on $M$. Then $F(T^+(M))$, together with the natural injection
$$J_M:M\hookrightarrow T^+(M)\hookrightarrow \sha^{\rm NC}_\frake(T^+(M)),$$
is the free \wmrb of weight $(\lambda,\kappa)$ on $M$.
\mlabel{coro:freerbm}
\end{coro}
\begin{proof}
The reduced tensor algebra on $M$ is the left adjoint of the forgetful functor from the category of algebras to that of vector spaces. By Theorem~\mref{thm:free}, the free \wmrb on $T^+(M)$ is the left adjoint of the forgetful functor the category of \wmrbs to that of algebras.  Since the composition of two
left adjoint functors is the left adjoint functor of the composition of the above two forgetful functors,   $\sha^{\rm NC}_\frake(T^+(M))$ is the free \wmrb on the vector space $M$.
\end{proof}
\subsection{The proof of Theorem~\mref{thm:free}}
We first list two easily verified properties for later applications. 
\begin{lemma}
Let $u,v,w\in \frakX_{\infty}$. Then
\begin{enumerate}
\item \mlabel{it:ht}
$h(u)=h(u \dg v)$ and $t(v)=t(u \dg v)$.
\item \mlabel{it:concat}
If $t(u)\ne h(v)$, then $u \dg v=u v$.
\item \mlabel{it:twoasso}
If $t(u)\ne h(v)$, then for all $w\in \frakX_{\infty}$,
\begin{equation}\mlabel{eq:twoa}
(uv)\dg w=u(v \dg w)\ , \quad w\dg (u v)=(w\dg u)v.
\end{equation}
\end{enumerate}
\mlabel{lem:pro}
\end{lemma}

\begin{lemma}Let $u,v,w\in\frakX_\infty$.
	If the associative law
	$$(u \dg v)\dg w=u \dg (v \dg w)$$
	holds for the case of $\bre(u)=\bre(v)=\bre(w)=1$, then it holds for all $u,v,w\in \frakX_{\infty}$.
	\mlabel{lem:ass}
\end{lemma}

We now give the proof of Theorem~\mref{thm:free}.
\begin{proof}[Proof of Theorem~\mref{thm:free}]
	(\mref{it:alg})
	To prove the associativity of $\dg$, it suffices to prove that
	\begin{equation}
		(u\dg v)\dg w=u\dg (v\dg w),\quad  u,v,w\in\frakX_\infty.
		\mlabel{eq:ass}
	\end{equation}
	The proof proceeds by induction on the sum $n:=\dep(u)+\dep(v)+\dep(w)\geq 0$ with $u,v,w\in\frakX_\infty$.
	For $n=0$, we have $\dep(u)=\dep(v)=\dep(w)=0$, and so $u, v, w\in \frakX_0$. By the definition of $\dg$, we have
	$$(u \dg v)\dg w=uvw=u \dg (v \dg w).$$
	For a given $k\geq 0$, assume that Eq.~(\mref{eq:ass}) holds for $0\le n\le k$, and consider the case of $n=k+1\geq 1$.
	If $t(u)\ne h(v)$, then by Eq.~\meqref{eq:twoa},
	$$(u\dg v)\dg w=(uv)\dg w=u(v\dg w)=u \dg (v\dg w).$$
	Similarly, if $t(v)\ne h(w)$, then by Eq.~\meqref{eq:twoa} again,
	$$(u \dg v)\dg w=(u \dg v)w=u\dg (vw)=u\dg (v\dg w).$$
	It remains to verify Eq.~(\mref{eq:ass}) when $t(u)=h(v)$ and $t(v)=h(w)$.
	By Lemma~\mref{lem:ass}, we only need to verify Eq.~\meqref{eq:ass} under the conditions that $\bre(u)=\bre(v)=\bre(w)=1$, $t(u)=h(v)$ and $ t(v)=h(w)$. Now suppose that  $\bre(u)=\bre(v)=\bre(w)=1$.  Thus $u,v,w$ are in $X\cup \lf \frakX_\infty\rf$. If one of $u,v,w$ is in $X$, then from $t(u)=h(v)$ and $t(v)=h(w)$ we conclude that $u,v,w$ are all in $X$, in which case Eq.~\meqref{eq:ass} is just the associativity of $A$. 
	
In the remaining case when $u,v,w$ are in $\lf \frakX_\infty \rf$, we let $u:=\lf\bar{u}\rf, v:=\lf\bar{v}\rf$, and $w:=\lf\bar{w}\rf$ for $\bar{u}, \bar{v}, \bar{w}\in \frakX_{\infty}$.
Then we have
\begin{align*}
&(u \dg v)\dg w\\
={}& (\lf\bar{u}\rf\dg \lf\bar{v}\rf)\dg\lf\bar{w}\rf\\
={}&(\lf\bar{u} \dg \lf\bar{v}\rf\rf+\lf\lf\bar{u}\rf \dg \bar{v}\rf+\lambda\lf\bar{u} \dg \bar{v}\rf+\kappa\bar{u} \dg \bar{v})\dg\lf\bar{w}\rf\\
={}& \lf\bar{u} \dg \lf\bar{v}\rf\rf\dg\lf\bar{w}\rf+\lf\lf\bar{u}\rf \dg \bar{v}\rf\dg\lf\bar{w}\rf+\lambda\lf\bar{u} \dg \bar{v}\rf\dg\lf\bar{w}\rf+\kappa(\bar{u} \dg \bar{v})\dg\lf\bar{w}\rf\\
={}&\lf(\bar{u} \dg \lf\bar{v}\rf)\dg\lf\bar{w}\rf\rf+\lf\lf\bar{u} \dg \lf\bar{v}\rf\rf\dg\bar{w}\rf+\lambda\lf(\bar{u} \dg \lf\bar{v}\rf)\dg\bar{w}\rf+\kappa(\bar{u} \dg \lf\bar{v}\rf)\dg\bar{w}\\
&+\lf(\lf\bar{u}\rf \dg \bar{v})\dg\lf\bar{w}\rf\rf+\lf\lf\lf\bar{u}\rf \dg \bar{v}\rf\dg\bar{w}\rf+\lambda\lf(\lf\bar{u}\rf \dg \bar{v})\dg\bar{w}\rf+\kappa(\lf\bar{u}\rf \dg \bar{v})\dg\bar{w}\\
&+\lambda\lf(\bar{u} \dg \bar{v})\dg\lf\bar{w}\rf\rf+\lambda\lf\lf\bar{u} \dg \bar{v}\rf\dg\bar{w}\rf+\lambda^2\lf(\bar{u} \dg \bar{v})\dg\bar{w}\rf+\lambda\kappa(\bar{u} \dg \bar{v})\dg\bar{w}\\
&+\kappa(\bar{u} \dg \bar{v})\dg\lf\bar{w}\rf\quad(\text{by Eq.~\meqref{eq:de}})\\
={}&\lf\bar{u} \dg (\lf\bar{v}\rf\dg\lf\bar{w}\rf)\rf+\lf\lf\bar{u} \dg \lf\bar{v}\rf\rf\dg\bar{w}\rf+\lambda\lf\bar{u} \dg (\lf\bar{v}\rf\dg\bar{w})\rf+\kappa\bar{u} \dg (\lf\bar{v}\rf\dg\bar{w})\\
&+\lf\lf\bar{u}\rf \dg (\bar{v}\dg\lf\bar{w}\rf)\rf+\lf\lf\lf\bar{u}\rf \dg \bar{v}\rf\dg\bar{w}\rf+\lambda\lf\lf\bar{u}\rf \dg (\bar{v}\dg\bar{w})\rf+\kappa\lf\bar{u}\rf \dg (\bar{v}\dg\bar{w})\\
&+\lambda\lf\bar{u} \dg (\bar{v}\dg\lf\bar{w}\rf)\rf+\lambda\lf\lf\bar{u} \dg \bar{v}\rf\dg\bar{w}\rf+\lambda^2\lf\bar{u} \dg (\bar{v}\dg\bar{w})\rf+\lambda\kappa\bar{u} \dg (\bar{v}\dg\bar{w})\\
&+\kappa\bar{u} \dg (\bar{v}\dg\lf\bar{w}\rf)\quad(\text{by the induction hypothesis})\\
={}&\lf\bar{u} \dg \underline{(\lf\bar{v}\rf\dg\lf\bar{w}\rf)}\rf+\lf\lf\bar{u} \dg \lf\bar{v}\rf\rf\dg\bar{w}\rf+\lambda\lf\bar{u} \dg (\lf\bar{v}\rf\dg\bar{w})\rf+\kappa\bar{u} \dg (\lf\bar{v}\rf\dg\bar{w})\\
&+\lf\lf\bar{u}\rf \dg (\bar{v}\dg\lf\bar{w}\rf)\rf+\lf\lf\lf\bar{u}\rf \dg \bar{v}\rf\dg\bar{w}\rf+\lambda\lf\lf\bar{u}\rf \dg (\bar{v}\dg\bar{w})\rf+\kappa\lf\bar{u}\rf \dg (\bar{v}\dg\bar{w})\\
&+\lambda\lf\bar{u} \dg (\bar{v}\dg\lf\bar{w}\rf)\rf+\lambda\lf\lf\bar{u} \dg \bar{v}\rf\dg\bar{w}\rf+\lambda^2\lf\bar{u} \dg (\bar{v}\dg\bar{w})\rf+\lambda\kappa\bar{u} \dg (\bar{v}\dg\bar{w})\\
&+\kappa\bar{u} \dg (\bar{v}\dg\lf\bar{w}\rf)\\
={}&\lf\bar{u} \dg \lf\bar{v}\dg\lf\bar{w}\rf\rf\rf+\lf\bar{u} \dg \lf\lf\bar{v}\rf\dg\bar{w}\rf\rf+\lambda\lf\bar{u} \dg \lf\bar{v}\dg\bar{w}\rf\rf+\kappa\lf\bar{u} \dg (\bar{v}\dg\bar{w})\rf\\
&+\lf\lf\bar{u} \dg \lf\bar{v}\rf\rf\dg\bar{w}\rf+\lambda\lf\bar{u} \dg (\lf\bar{v}\rf\dg\bar{w})\rf+\kappa\bar{u} \dg (\lf\bar{v}\rf\dg\bar{w})\\
&+\lf\lf\bar{u}\rf \dg (\bar{v}\dg\lf\bar{w}\rf)\rf+\lf\lf\lf\bar{u}\rf \dg \bar{v}\rf\dg\bar{w}\rf+\lambda\lf\lf\bar{u}\rf \dg (\bar{v}\dg\bar{w})\rf+\kappa\lf\bar{u}\rf \dg (\bar{v}\dg\bar{w})\\
&+\lambda\lf\bar{u} \dg (\bar{v}\dg\lf\bar{w}\rf)\rf+\lambda\lf\lf\bar{u} \dg \bar{v}\rf\dg\bar{w}\rf+\lambda^2\lf\bar{u} \dg (\bar{v}\dg\bar{w})\rf+\lambda\kappa\bar{u} \dg (\bar{v}\dg\bar{w})\\
&+\kappa\bar{u} \dg (\bar{v}\dg\lf\bar{w}\rf)
\end{align*}
by applying Eq.~\meqref{eq:de} to the above underlined factor. On the other hand, we get
\begin{align*}
&u \dg(v \dg w)\\
={}& \lf\bar{u}\rf \dg (\lf\bar{v}\rf \dg \lf\bar{w}\rf)\\
={}& \lf\bar{u}\rf \dg (\lf\bar{v} \dg \lf\bar{w}\rf\rf+\lf\lf\bar{v}\rf \dg \bar{w}\rf+\lambda\lf\bar{v} \dg \bar{w}\rf+\kappa\bar{v} \dg \bar{w})\\
={}&\lf\bar{u}\rf \dg \lf\bar{v} \dg \lf\bar{w}\rf\rf+\lf\bar{w}\rf \dg \lf\lf\bar{v}\rf \dg \bar{w}\rf+\lambda\lf\bar{u}\rf \dg \lf\bar{v} \dg \bar{w}\rf+\kappa\lf\bar{u}\rf \dg(\bar{v} \dg \bar{w})\\
={}&\lf\bar{u} \dg \lf\bar{v} \dg \lf\bar{w}\rf\rf\rf+\lf\lf\bar{u}\rf \dg (\bar{v} \dg \lf\bar{w}\rf)\rf+\lambda\lf\bar{u} \dg (\bar{v} \dg \lf\bar{w}\rf)\rf+\kappa\bar{u} \dg (\bar{v} \dg \lf\bar{w}\rf)\\
&+\lf\bar{u} \dg \lf\lf\bar{v}\rf \dg \bar{w}\rf\rf+\lf\underline{(\lf\bar{u}\rf \dg\lf\bar{v}\rf)} \dg \bar{w}\rf+\lambda\lf\bar{u} \dg (\lf\bar{v}\rf \dg \bar{w})\rf+\kappa\bar{u} \dg (\lf\bar{v}\rf \dg \bar{w})\\
&+\lambda\lf\bar{u} \dg \lf\bar{v} \dg \bar{w}\rf\rf+\lambda\lf\lf\bar{u}\rf \dg (\bar{v} \dg \bar{w})\rf+\lambda^2\lf\bar{u} \dg (\bar{v} \dg \bar{w})\rf+\lambda\kappa\bar{u} \dg (\bar{v} \dg \bar{w})\\
&+\kappa\lf\bar{u}\rf \dg (\bar{v} \dg \bar{w})\quad(\text{by the induction hypothesis and Eq.~\meqref{eq:de}}) \\
={}&\lf\bar{u} \dg \lf\bar{v} \dg \lf\bar{w}\rf\rf\rf+\lf\lf\bar{u}\rf \dg (\bar{v} \dg \lf\bar{w}\rf)\rf+\lambda\lf\bar{u} \dg (\bar{v} \dg \lf\bar{w}\rf)\rf+\kappa\bar{u} \dg (\bar{v} \dg \lf\bar{w}\rf)\\
&+\lf\bar{w} \dg \lf\lf\bar{w}'\rf \dg \bar{w}''\rf\rf+\lf\lf\bar{u} \dg\lf\bar{v}\rf\rf \dg \bar{w}\rf+\lf\lf\lf\bar{u}\rf \dg\bar{v}\rf \dg \bar{w}\rf+\lambda\lf\lf\bar{u} \dg\bar{v}\rf \dg \bar{w}\rf\\
&+\kappa{\lf(\bar{u} \dg\bar{v}) \dg \bar{w}}\rf+\lambda\lf\bar{u} \dg (\lf\bar{v}\rf \dg \bar{w})\rf+\kappa\bar{u} \dg (\lf\bar{v}\rf \dg \bar{w})\\
&+\lambda\lf\bar{u} \dg \lf\bar{v} \dg \bar{w}\rf\rf+\lambda\lf\lf\bar{u}\rf \dg (\bar{v} \dg \bar{w})\rf+\lambda^2\lf\bar{u} \dg (\bar{v} \dg \bar{w})\rf+\lambda\kappa\bar{u} \dg (\bar{v} \dg \bar{w})\\
&+\kappa\lf\bar{u}\rf \dg (\bar{v} \dg \bar{w})\quad(\text{by applying Eq.~\meqref{eq:de} to the above underlined factor})\\
={}&\lf\bar{u} \dg \lf\bar{v} \dg \lf\bar{w}\rf\rf\rf+\lf\lf\bar{u}\rf \dg (\bar{v} \dg \lf\bar{w}\rf)\rf+\lambda\lf\bar{u} \dg (\bar{w} \dg \lf\bar{w}\rf)\rf+\kappa\bar{u} \dg (\bar{v} \dg \lf\bar{w}\rf)\\
&+\lf\bar{u} \dg \lf\lf\bar{v}\rf \dg \bar{w}\rf\rf
+\lf\lf\bar{u} \dg\lf\bar{v}\rf\rf \dg \bar{w}\rf+\lf\lf\lf\bar{u}\rf \dg\bar{v}\rf \dg \bar{w}\rf+\lambda\lf\lf\bar{u} \dg\bar{v}\rf \dg \bar{w}\rf\\
&+\kappa\lf\bar{u}\dg(\bar{v} \dg \bar{w})\rf+\lambda\lf\bar{u} \dg (\lf\bar{v}\rf \dg \bar{w})\rf+\kappa\bar{u} \dg (\lf\bar{w}\rf \dg \bar{w})\\
&+\lambda\lf\bar{u} \dg \lf\bar{v} \dg \bar{w}\rf\rf+\lambda\lf\lf\bar{u}\rf \dg (\bar{v} \dg \bar{w})\rf+\lambda^2\lf\bar{u} \dg (\bar{v} \dg \bar{w})\rf+\lambda\kappa\bar{u} \dg (\bar{v} \dg \bar{w})\\
&+\kappa\lf\bar{u}\rf \dg (\bar{v} \dg \bar{w})
\end{align*}
by  the induction hypothesis. Then the $i$-th term in the former expansion of $(u \dg v)\dg w$ coincides with the $\sigma(i)$-th term in the latter expansion of $u \dg (v \dg w)$. Here the order 16 permutation $\sigma\in \sum_{16}$ is
\begin{equation*}
\begin{pmatrix}
i\\ \sigma(i)
\end{pmatrix}
=
\left ( \begin{array}{cccccccccccccccc}
1 & 2 & 3 & 4 & 5 & 6 & 7 & 8 & 9 & 10 & 11 & 12 & 13 & 14 & 15 & 16 \\
1 & 5 & 12 & 9 & 6 & 10 & 11 & 2 & 7 & 13 & 16 & 3 & 8 & 14 & 15 & 4
\end{array} \right).
\end{equation*}
Thus we have completed the inductive proof of the associativity of $\dg$.
	
\smallskip
\noindent
(\mref{it:erb}) follows from Eq.~\meqref{eq:de}.
	
\smallskip
\noindent
(\mref{it:ferb})
Let $(R,\ast, P)$ be an \wmrb, and let $f:A\to R$ be an algebra homomorphism. We will define  an \wmrb homomorphism $\tilde{f}:\ce\to R$ by defining $\tilde{f}(w)$ for all basis elements $w\in \frakX_{\infty}$. This is carried out by induction on the depth of $w$, $\dep(w)\geq 0$.  If $\dep(w)=0$, then $w\in \frakX_{0}=S(X)$.  If $w=w_1\cdots w_m\in S(X)$, where $w_i\in X$ for $i=1,\cdots, m$, then define
\begin{equation}\mlabel{eq:baf}
		\tilde{f}(w)=f(w_1)\ast\cdots\ast f(w_m).
\end{equation}
	
For a given $k\geq 0$,  assume that $\tilde{f}$ has been defined for all $w\in \frakX_{\infty}$ with $\dep(w)\leq k$, and consider $w\in \frakX_{\infty}$ with $\dep(w)=k+1\geq 1$. So $w\in \frakX_{k+1}\backslash \frakX_k$. By the definition of $\frakX_{k+1}$, we have
	\begin{align*}
		\frakX_{k+1}={}& \Big(\bigcup_{r\geq 1}(S(X)\lf \frakX_{k}\rf)^r\Big)\bigcup\Big(\bigcup_{r\geq 0}(S(X)\lf \frakX_{k}\rf)^rS(X)\Big)\\
		&\bigcup\Big(\bigcup_{r\geq 1}(\lf \frakX_{k}\rf S(X))^r\Big)\bigcup\Big(\bigcup_{r\geq 0}(\lf \frakX_{k}\rf S(X))^r\lf \frakX_{k}\rf\Big).
	\end{align*}
We only consider the case when $w$ is in the first union component $\bigcup_{r\geq 1}(S(X)\lf \frakX_{n}\rf)^r$ above. The other cases are similar. Then we can write
$$w=\prod_{i=1}^{r}(w_{2i-1}\lf w_{2i}\rf),$$
where $w_{2i-1}\in S(X)$ and $w_{2i}\in \frakX_{k}, 1\le i\le r$. By the definitions of  $\dg$ and $P_e$, we get
$$w=\dg_{i=1}^{r}\Big(w_{2i-1} \dg P_e(\frakX_{2i})\Big).$$
Define
\begin{equation}\mlabel{eq:dfnf}
		\tilde{f}(w)=\ast_{i=1}^{r}\Big(\tilde{f}(w_{2i-1})\ast P\big(\tilde{f}(w_{2i})\big)\Big),
\end{equation}
where $\tilde{f}(w_{2i-1})$ has been defined by Eq.~\meqref{eq:baf}, and $\tilde{f}(w_{2i})$ are well-defined by the induction hypothesis. Thus this completes the induction. Extending by $\bfk$-linearity, we obtain a linear map $\tilde{f}:\ce\to R$, which satisfies $\tilde{f}\circ j_A= f$ by Eq.~\meqref{eq:baf}.
	
Note that for all $w\in \frakX_{\infty}$, $P_e(w)=\lf w\rf\in \frakX_{\infty}$. Then by the definition of $\tilde{f}$, we have
\begin{equation}\mlabel{eq:commgrba}
		\tilde{f}(P_e(w))=P(\tilde{f}(w)).
\end{equation}
For all $w\in\frakX_\infty$, let $w=w_1\cdots w_m$ be the alternating decomposition of $w$.  By the definition of $\tilde{f}$ again,
\begin{equation}\mlabel{eq:dfns}
		\tilde{f}(w)=\tilde{f}(w_1)\ast\cdots \ast \tilde{f}(w_m).
\end{equation}
Next we shall verify that the map $\tilde{f}$ defined above is an algebra homomorphism. It suffices to show that
\begin{equation}\mlabel{eq:ahomo}
		\tilde{f}(u \dg v)=\tilde{f}(u)\ast\tilde{f}(v),\quad u,v\in \frakX_{\infty}.
\end{equation}
Proceed by induction on the sum $n:=d(u)+d(v)\geq 0$. When $n=0$, then $u,v\in \frakX_{0}$. By  Eq.~\meqref{eq:baf}, we get
\begin{equation}
		\tilde{f}(u \dg v)=\tilde{f}(u)\ast \tilde{f}(v).
		\mlabel{eq:ast}
\end{equation}
For a given $k\geq 0$,	assume that Eq.~\meqref{eq:ahomo} holds for $u,v\in \frakX_{\infty}$ with $0\le n\le k$.  Now consider the case of $n=k+1\geq 1$. Let $u=u_1\cdots u_m $ and $v=v_1\cdots v_{\ell}$ be the alternating decompositions of $u,v$, respectively. If $u_m, v_1$ are in $\frakX_0$, then $\tilde{f}(u_m\dg v_1)= \tilde{f}(u_m)\ast \tilde{f}(v_1)$ by Eq.~\meqref{eq:ast}. For the other cases,  by Eq.~(\mref{eq:de}), we have
\begin{equation*}
\tilde{f}(u_m \dg v_1) = \left\{\begin{array}{ll} \tilde{f}(\lf\bar{u}_m \dg \lf\bar{v}_1\rf+\lf\bar{u}_m\rf \dg \bar{v}_1+\lambda\bar{u}_m \dg \bar{v}_1\rf+\kappa\bar{u}_m \dg \bar{v}_1), & \text{if}\, u_m=\lf\bar{u}_m\rf, v_1=\lf\bar{v}_1\rf,\\
\tilde{f}(u_m v_1),  &\text{otherwise}.
\end{array} \right .
\end{equation*}
For the first subcase,  we have
\begin{align*}
\tilde{f}(u_m \dg v_1)={}&\tilde{f}\Big(\lf\bar{u}_m \dg \lf\bar{v}_1\rf+\lf\bar{u}_m\rf \dg \bar{v}_1+\lambda\bar{u}_m \dg \bar{v}_1\rf+\kappa\bar{u}_m \dg \bar{v}_1\Big)\\
={}& P\Big(\tilde{f}(\bar{u}_m \dg \lf\bar{v}_1\rf)+\tilde{f}(\lf\bar{u}_m\rf \dg \bar{v}_1)+\lambda \tilde{f}(\bar{u}_m \dg \bar{v}_1)\Big)+\kappa\tilde{f}(\bar{u}_m \dg \bar{v}_1)\quad(\text{by Eq.~\meqref{eq:commgrba}})\\
={}& P\Big(\tilde{f}(\bar{u}_m) \ast P(\tilde{f}(\bar{v}_1))+P(\tilde{f}(\bar{u}_m)) \ast \tilde{f}(\bar{v}_1)+\lambda \tilde{f}(\bar{u}_m) \ast \tilde{f}(\bar{v}_1)\Big)+\kappa\tilde{f}(\bar{u}_m)\ast \tilde{f}(\bar{v}_1)\\
&\quad(\text{by the induction hypothesis})\\
={}& P\big(\tilde{f}(\bar{u}_m)\big)\ast P\big(\tilde{f}(\bar{v}_1)\big)\quad(\text{by $P$ being an \wmrbo})\\
={}& \tilde{f}(\lf\bar{u}_m\rf)\ast \tilde{f}(\lf\bar{v}_1\rf)\quad(\text{by Eq.~\meqref{eq:commgrba}})\\
={}& \tilde{f}(u_m)\ast \tilde{f}(v_1).
\end{align*}
For the second subcase, by Eq.~\meqref{eq:dfns}, we get $\tilde{f}(u_m \dg v_1)=\tilde{f}(u_m v_1)=\tilde{f}(u_m)\ast\tilde{f}(v_1)$.
Therefore, we have 
\begin{align*}
\tilde{f}(u \dg v)={}& \tilde{f}\big(u_1\cdots u_{m-1}(u_m \dg v_1)v_2\cdots v_{n}\big)\\
={}& \tilde{f}(u_1)\ast\cdots\ast\tilde{f}(u_{m-1})\ast\tilde{f}(u_m \dg v_1)\ast\tilde{f}(v_2)\ast\cdots\ast\tilde{f}(v_{n})\quad(\text{by Eq.~\meqref{eq:dfns}})\\
={}& \tilde{f}(u_1)\ast\cdots\ast\tilde{f}(u_{m-1})\ast\tilde{f}(u_m)\ast\tilde{f}(v_1)\ast\tilde{f}(v_2)\ast\cdots\ast\tilde{f}(v_{n})\\
={}& \tilde{f}(u)\ast\tilde{f}(v),
\end{align*}
as desired.
	
Finally, we verify the uniqueness of $\tilde{f}$. Suppose that there exists another \wmrb homomorphism $\tilde{f'}$ suth that $\tilde{f'}\circ j_A=f$. We shall prove that
\begin{equation}\mlabel{eq:unique}
		\tilde{f}(w)=\tilde{f}(w), \ w\in \frakX_{\infty}.
\end{equation}
By induction on $n:=\dep(w)\geq 0$, for $n=0$, we have $w=w_1\cdots w_m\in \frakX_{0}$ for $w_i\in X$ and $i=1,\cdots,m$, and  so
\begin{eqnarray*}
\tilde{f'}(w)&=&\tilde{f'}\Big(w_1\dg\cdots\dg w_m\Big)\\
&=&\tilde{f'}(w_1)\ast\cdots\ast \tilde{f'}( w_m)\\
&=&f(w_1)\ast\cdots\ast f( w_m)\quad(\text{by $\tilde{f'}\circ j_X=f$})\\
&=&\tilde{f}(w).
\end{eqnarray*}
For a given $\ell\geq 0$, assume that Eq.~(\mref{eq:unique}) holds for all $w\in \frakX_\infty$ with $n\leq \ell$. Consider the case of $n=\ell+1$. Let $w=w_1\cdots w_m$ be the alternating decomposition of $w$.  If $w_i\in X$, then $\tilde{f'}(w_i)=\tilde{f}(w_i)$ by the case of $n=0$.  If $w_i\in \lf \frakX_\infty\rf$, then we write $w_i=\lf\bar{w}_i\rf$ with $\bar{w}_i\in \frakX_{\infty}$.  By the induction hypothesis,
$$\tilde{f'}(w_i)=P\big(\tilde{f'}(\bar{w}_i)\big)=P\big(\tilde{f}(\bar{w}_i)\big)=\tilde{f}(w_i).$$
Then we have
\begin{align*}
\tilde{f'}(w)={}& \tilde{f'}(w_1\cdots w_m)\\
={}& \tilde{f'}(w_1\dg\cdots\dg w_m)\\
={}& \tilde{f'}(w_1)\ast\cdots\ast\tilde{f'}(w_m)\\
={}& \tilde{f}(w_1)\ast\cdots\ast\tilde{f}(w_m)\\
={}& \tilde{f}(w_1\cdots w_m)\\
={}& \tilde{f}(w).
\end{align*}
Thus this completes the induction and so the proof of the uniqueness of $\tilde{f}$. Hence the proof of Theorem~\mref{thm:free} is finished.
\end{proof}


\section{Extended (tri)dendriform algebras as the companion structures of \wmrbs}
\mlabel{sec:wmrbet}

Let $P:R\to R$ be an \wmrbo. By Theorem~\mref{thm:erbt}, the three operations $\prp, \scp$ and $\obp$ given in Eq.~\meqref{eq:prp} satisfy  the \et relations from Eq.~\meqref{eq:e1} to Eq.~\meqref{eq:e7}. To give a precise formulation and answer to the converse question, we first introduce the general notions of tri-companion and di-companion of the operad of an operated algebra as in Definition~\mref{de:mapset} that encode all binary quadratic relations that can be possibly derived from the operated algebra. We then use the construction of free \wmrbs obtained the previous section to show that the operads of extended tridendriform algebras and extended dendriform algebras coincide respectively with the tri-companion and di-companion of the operad of \wmrbs. The di-companion of extended Rota-Baxter algebra leads to a classification of more general extended dendriform algebras. 

\subsection{A general formulation of companion structures of unary binary operads}
\mlabel{ss:ind}

We first recall the notion of a binary quadratic nonsymmetric operad  (called ns operad for short). 
For details on binary quadratic ns operads, see~\mcite{LV12,ZGG}.
\begin{defn}
		\begin{enumerate}
			\item
			A {\bf graded vector space} is a sequence  $\calp:=\{\calp_n\}_{n\geq 0}$ of vector spaces $\calp_n, n\geq 0$.
			\item
			A {\bf nonsymmetric (ns) operad} is a graded vector space $\calp=\{\calp_n\}_{n\geq 0}$ equipped with {\bf partial compositions}:
			\begin{equation}
				\circ_i:=\circ_{m,n,i}: \calp_m\ot \calp_n\longrightarrow \calp_{m+n-1}, \quad 1\leq i\leq m,
				\mlabel{eq:opc}
			\end{equation}
			such that, for $\eta\in\calp_\ell, \mu\in\calp_m$ and $\nu\in\calp_n$, the following relations hold.
			\begin{enumerate}
				\item[(i)] $
				(\eta \circ_i \mu)\circ_{i-1+j}\nu = \eta\circ_i (\mu\circ_j\nu), \quad 1\leq i\leq \ell, 1\leq j\leq m.$
				\mlabel{it:esc}
				\item[(ii)]$(\eta\circ_i\mu)\circ_{k-1+m}\nu =(\eta\circ_k\nu)\circ_i\mu, \quad
				1\leq i<k\leq \ell.$
				\mlabel{it:epc}
				\item[(iii)]
				There is an element $\id\in \calp_1$ such that $\id\circ \mu=\mu$ and $\mu\circ\id=\mu$ for $\mu\in \calp_n, n\geq 0$.
				\mlabel{it:id}
			\end{enumerate}
		\end{enumerate}
\end{defn}

A class of ns operads can be given more explicitly as follows. Here we follow the exposition in~\mcite{Gub}.
\begin{defn}
	\begin{enumerate}
		\item
		A ns operad $\calp=\{\calp_n\}$ is called {\bf binary} if $\calp_1=\bfk.\id$ and $\calp_n,
		n\geq 3$ are induced from $\calp_2$ by composition. Then in particular, for the free operad, we have
		\begin{equation}
			\calp_3=(\calp_2 \circ_1 \calp_2) \oplus (\calp_2\circ_2 \calp_2),
			\mlabel{eq:bq}
		\end{equation}
		which can be identified with $\calp_2^{\ot 2}\oplus \calp_2^{\ot 2}$.
		\item
		A binary ns operad $\calp$ is called {\bf quadratic} if
		all relations among the binary operations in $\calp_2$ are derived
		from $\calp_3$.
		\end{enumerate}
\end{defn}

A binary quadratic ns operad $\calp$
is determined by a pair $(\dfgen,\dfrel)$ where $\dfgen=\calp_2$, called
the {\bf generator space },
and $\dfrel$ is a subspace of $\dfgen^{\otimes 2} \oplus \dfgen^{\otimes 2}$,
called the {\bf relation  space}. More precisely, let $\calf(V)$ be the free ns operad on $V$. Then the binary, quadratic operad $\calp$ is the quotient of $\calf(V)$
$$\calp(\dfgen,\dfrel):=\calf(V)/(\dfrel),$$
where $(\dfrel)$ is the operadic ideal  of $\calf(V)$ generated by $\dfrel$. 

Since every element of $\dfgen^{\ot 2}$ is a sum $\sum\limits_{i=1}^k\dfoa_i\otimes \dfob_i$ with $\dfoa_i,\dfob_i\in V, 1\leq i\leq k$, every element of
		$\dfgen^{\otimes 2} \oplus \dfgen^{\otimes 2}$ can be written as a pair
		$$\Big (\sum_{i=1}^k\dfoa_i\otimes \dfob_i, \sum_{j=1}^m\dfoc_j\otimes \dfod_j \Big)$$
		with $\dfoa_i,\dfob_i,\dfoc_j,\dfod_j\in \dfgen, 1\leq i\leq k, 1\leq j\leq m, k, m\geq 1$.  For a given binary quadratic ns operad $\calp:=\calp(\dfgen,\dfrel)$, a vector space $A$ is called a {\bf $\calp$-algebra}
		if $A$ has binary operations (indexed by) $\dfgen$ and if,
		for
		$$\Big (\sum_{i=1}^k\dfoa_i\otimes \dfob_i, \sum_{j=1}^m\dfoc_j\otimes \dfod_j \Big)
		\in \dfrel \subseteq
		\dfgen^{\otimes 2} \oplus \dfgen^{\otimes 2},$$
		we have
		\begin{equation}
			\sum_{i=1}^k(x\dfoa_i y) \dfob_i z = \sum_{j=1}^m x \dfoc_j (y \dfod_j z), \quad\ x,y,z\in A.
			\mlabel{eq:rel}
		\end{equation}

In particular, for fixed $\lambda, \kappa\in \bfk$,  the {\bf ns operad of extended tridendriform algebras} is defined by the the quotient
$$\oetd:=\oetd_{\lambda,\kappa}:=\calp(V,\dfrel_{\rm ETD}),$$ 
where $V$ is the generator space $\bfk\{\prec, \succ, \ob\}$ spanned by $\prec, \succ, \ob$, and $\dfrel_{\rm ETD}:=\dfrel_{\rm ETD,\lambda,\kappa}\subseteq V^{\ot 2}\oplus V^{\ot 2}$ is the relation space spanned by 
\begin{eqnarray}
	\begin{split}
		&(\prec \ot \prec,\prec\ot \star_\lambda+{\kappa}\ob\ot \ob),\;(\succ \ot \prec,\succ\ot \prec), \;(\star_\lambda\ot \succ +\kappa\ob\ot \ob, \succ\ot \succ),\\
		&(\succ\ot\ob, \succ\ot\ob),\; (\prec\ot\ob, \ob\ot\succ),\; (\ob\ot\prec, \ob\ot\prec),\;(\ob\ot\ob, \ob\ot\ob), 
	\end{split}	
	\mlabel{eq:wdn}
\end{eqnarray}
where $\star_\lambda=\prec+\succ+\lambda\,\ob\,$.

We now give a general framework to formulate the derived structures of an algebra with equipped with linear operators, such derivations and Rota-Baxter operators. 

Recall that a {\bf unary binary operad} is an operad with only unary and binary operations~\mcite{ZGG}. More precisely, it is a quotient of the free ns operad $\calf(W)$ where the graded space $W=W_1\oplus W_2$ is concentrated in degrees one and two. So a unary binary operad is of the form $\calp=\calp(W,S)=\calf(W)/(S)$, where $(S)$ is the operadic ideal generated by a subset $S\subseteq\calf(W)$.

Let $\Omega$ and $\Theta$ be nonempty sets. Let 
$$\rho_\Omega:=\{\rho_\omega\,|\, \omega\in \Omega\}$$ be a linear basis of $W_1$, and let $$\mu_\Theta:=\{\mu_\theta\,|\,\theta\in \Theta\}$$ 
be a linear basis of $W_2$. 
For a given uniary binary operad 
$$\calq:=\calp(W,S)=\calf(W)/(S),$$
define the graded vector space concentrated at degree two: $$\genpost:=\genpost_{\Omega,\Theta}=\bfk\{\prec_{\omega,\theta},\succ_{\omega,\theta}, \ob_{\theta}\,|\,\omega\in \Omega, \theta\in \Theta\},$$
and define the operad homomorphism, called the {\bf tri-companion homomorphism}  
\begin{equation} \mlabel{eq:genindtri}
\morpost:=\morpost_{\calq}: \calf(\genpost)\to \calq, \quad
\left\{\begin{array}{l} \prec_{\omega,\theta}\mapsto \mu_\theta(\rho_\omega\ot \id), \\ \succ_{\omega,\theta}\mapsto \mu_\theta(\id\ot \rho_\omega),\\
 \ob_{\theta}\mapsto \mu_\theta,
 \end{array}\right. \theta\in \Theta, \omega\in \Omega.
\end{equation}

Denote 
\begin{equation} \mlabel{eq:genindreltri}
\relpost:=\relpost{}_\calq:=\ker \morpost \cap \Big(\genpost^{\ot 2}\oplus \genpost^{\ot 2}\Big).
\end{equation}
\begin{defn} \mlabel{de:postind}
The {\bf tri-companion binary quadratic operad} or simply the {\bf tri-companion} of the unary binary operad $\calq=\calp(W,S)$ is defined to be  the binary quadratic ns operad given by the quotient 
\begin{equation}\mlabel{eq:genindquottri} \oppost:=\calf(\genpost)/(\relpost).
\end{equation}

\end{defn}

Similarly, define the graded vector space concentrated at degree two: $$\genpre:=\genpre_{\Omega,\Theta}=\bfk\{\prec_{\omega,\theta},\succ_{\omega,\theta}\,|\,\omega\in \Omega, \theta\in \Theta\}.$$
Define an operad homomorphism, called the {\bf di-companion homomorphism}  
\begin{equation} \mlabel{eq:geninddi}
	\morpre:=\morpre_{\calq}: \calf(\genpre)\to \calq, \quad
	\left\{\begin{array}{l} \prec_{\omega,\theta}\mapsto \mu_\theta(\rho_\omega\ot \id), \\ \succ_{\omega,\theta}\mapsto \mu_\theta(\id\ot \rho_\omega),
	\end{array}\right. \theta\in \Theta, \omega\in \Omega.
\end{equation}
Denote 
\begin{equation} \mlabel{eq:genindreldi}
	\relpre:=\relpre{}_\calq:=\ker \morpre \cap \Big(\genpre^{\ot 2}\oplus \genpre^{\ot 2}\Big).
\end{equation}
\begin{defn} \mlabel{de:preind}
The {\bf di-companion binary quadratic operad} or simply the {\bf di-companion} of $\calq=\calp(W,S)$ is defined to be  the binary quadratic ns operad given by the quotient \begin{equation}\mlabel{eq:genindquotdi} \oppre:=\calf(\genpre)/(\relpre).
\end{equation}
\end{defn}

\begin{remark}
\begin{enumerate}
\item 
Formally, the di-companion can be regarded as the tri-companion when the multiplications $\ob_\theta$ are taken to be zero. But the two structures can be quite different. To illustrate this phenomenon, we give separate discussions of extended tridendriform algebras and extended dendriform algebras in the next two subsections. 	
\item
Similar calculations as in the next two subsections show that the operads of (tri)dendriform algebras are the tri-(resp. di-)companions of the operad of Rota-Baxter algebras. Other derived structures (of nonsymmetric operads) can also be put into this framework. See~\mcite{LG} for the case of Nijenhuis algebras.  
\item 
A similar approach can be taken for symmetric operads. See~\mcite{KSO} for the case when the space of unary operations is given by a derivation. 
\end{enumerate}
\end{remark}

\subsection{Extended tridendriform algebras as the tri-companion of extended Rota-Baxter algebras}
\mlabel{ss:eterb}

Let $\oerb:=\oerb_{\lambda,\kappa}$ be the {\bf ns operad of extended Rota-Baxter algebras of weight $(\lambda,\kappa)$}. It can be realized as the quotient 
$\oerb=\calp(W,\dfrel_{\rm ERB,\lambda,\kappa})$ where $W=W_1\oplus W_2$ is the generator space  with $W_1=\bfk\,\rho$ and $W_2=\bfk \mu$, and $\dfrel_{\rm ERB,\lambda,\kappa}$ is the relation space spanned by the associativity of $\mu$ and the extended Rota-Baxter axiom
$$\mu(\rho\ot \rho)-\rho(\mu(\id\ot \rho))-\rho(\mu(\rho\ot \id))-\lambda\rho(\mu(\id\ot \id))-\kappa \mu(\id\ot \id)$$
coming from Eq.~\meqref{eq:grb}. 

\begin{theorem}\mlabel{thm:owdn}
Fix $\lambda,\kappa\in \bfk$. The operad $\oetd$ is the tri-companion $\mathrm{TC}(\oerb)$ of the operad of extended Rota-Baxter algebras. 
More precisely, let $V=\bfk\{\prec,\succ,\ob\}$ be the vector space with basis $\{\prec,\succ,\ob\}$. For the operadic homomorphism 
\begin{equation}
 \Phi: \calf(V) \to \oerb, \quad	\prec\,\mapsto \mu(\id\ot \rho), \quad \succ\,\mapsto \mu(\rho\ot \id), \quad \ob \mapsto \mu,
	\mlabel{eq:ofreerb}
\end{equation}
we have  
\begin{equation}
	\ker \Phi \cap (V^{\ot 2}\oplus V^{\ot 2}) = \dfrel_{\rm ETD}.
\end{equation}
\end{theorem}

\begin{proof}
We first prove $\ker \Phi \cap (V^{\ot 2}\oplus V^{\ot 2}) \subseteq  \dfrel_{\rm ETD}.$

With $V=\bfk \{\prec, \succ, \ob\}$, we have
	$$V^{\ot 2} \oplus V^{\ot 2} = \bigoplus\limits_{\dfop_1,\dfop_2,\dfop_3,\dfop_4\in \{\prec,\succ,\ob\}} \bfk (\dfop_1\ot\dfop_2, \dfop_3\ot\dfop_4).$$
	Then every element $r$ of $\dfgen^{\otimes 2} \oplus \dfgen^{\otimes 2}$ is a linear combination
	\begin{equation}
\begin{split}		
	r=&a_1(\prec\ot \prec,0)+a_2(\prec\ot\succ,0)+a_3(\prec\ot \ob,0)\\
		&+b_1(\succ\ot \prec,0)+b_2(\succ\ot \succ,0)+b_3(\succ\ot\ob,0)\\
		&+c_1(\ob\,\ot \prec,0)+c_2(\ob\,\ot\succ,0)+c_3(\ob\ot\ob,0)\\
		&+d_1(0,\prec\ot \prec)+d_2(0,\prec\ot \succ)+d_3(0,\prec\ot \ob)\\
		&+e_1(0,\succ\ot \prec)+e_2(0,\succ\ot \succ)+e_3(0,\succ\ot\ob)\\
		&+f_1(0,\ob\ot\prec)+f_2(0,\ob\ot\succ) +f_3(0,\ob\ot\ob),
\end{split}
\mlabel{eq:rrel}
	\end{equation}
	where the coefficients $a_i,b_i,c_i,d_i,e_i,f_i$ are in $\bfk$ for $1\leq i\leq 3$.
	
Note that, for $r$ in the form of Eq.~\meqref{eq:rrel}, its image $\Phi(r)$ is obtained by simply replacing $\prec$, $\succ$ and $\ob$ by $\mu(\id\ot \rho), \mu(\rho \ot \id)$ and $\mu$ respectively. 

Now suppose that $r\in V^{\ot 2}\oplus V^{\ot 2}$ is in $\ker \Phi$. Then $\Phi(r)=0$ in $\oerb$ means that $\Phi(r)$ is an identity that is satisfied by every extended Rota-Baxter algebra $(R,P)$ of weight $\lambda$, where  $\mu(\id\ot \rho), \mu(\rho \ot \id)$ and $\mu$ correspond to the multiplications  
$$\mu(\id\ot \rho)(x\ot y)=x\prec_P y=xP(y), \quad \mu(\rho\ot \id)(x\ot y)=x\succ_P y=P(x)y, \quad \mu(x\ot y)=xy, \quad x, y\in R.$$

Therefore, for all $x,y,z\in R$, we have 
	\begin{eqnarray*}
		&&a_1(x\prec_P y)\prec_P z+a_2(x\prec_P y)\succ_P z +a_3(x\prec_P y)\obp z \\
		&& +b_1(x\succ_P y)\prec_P z+b_2(x\succ_P y)\succ_P z+b_3(x\succ_P y)\obp z\\
		&&+c_1(x\obp y)\prec_P z+c_2(x\obp y)\succ_P z+c_3(x\obp y)\obp z\\
		&&-d_1 x\prp(y\prp z)-d_2 x\prp(y\scp z)-d_3 x\prp(y\obp z)\\
		&&-e_1 x\scp(y\prp z)-e_2 x\scp(y\scp z)-e_3x\scp(y\obp z) \\
		&&-f_1x \obp(y\prp z)-f_2x \obp (y\scp z)-f_3x \obp (y\obp z)
		=0.
	\end{eqnarray*}
That is,
\begin{eqnarray*}
		&&a_1xP(y)P(z)+a_2P(xP(y))z+ a_3xP(y)z+b_1P(x)yP(z)
		+b_2P(P(x)y)z+ b_3P(x)yz\\
		&&+ c_1xyP(z)+ c_2P(xy)z
		+ c_3xyz-d_1xP(yP(z))-d_2xP(P(y)z)
		-d_3xP(yz)\\
		&&-e_1P(x)yP(z)-e_2P(x)P(y)z-e_3P(x)yz-f_1xyP(z)-f_2xP(y)z-f_3xyz
		=0.
	\end{eqnarray*}
	By $P$ being an \wmrbo, we obtain
	\begin{eqnarray*}
		&&a_1xP(yP(z))+a_1xP(P(y)z)+\lambda a_1xP(yz)+\kappa a_1xyz+a_2P(xP(y))z+ a_3xP(y)z\\
		&&+b_1P(x)yP(z)
		+b_2P(P(x)y)z+ b_3P(x)yz
		+ c_1xyP(z)+ c_2P(xy)z
		+ c_3xyz\\
		&&-d_1xP(yP(z))-d_2xP(P(y)z)
		-d_3xP(yz)-e_1P(x)yP(z)-e_2P(xP(y))z-e_2P(P(x)y)z\\
		&&-\lambda e_2P(xy)z-\kappa e_2xyz-e_3P(x)yz-f_1xyP(z)-f_2xP(y)z-f_3xyz
		=0.
	\end{eqnarray*}
Collecting like terms, we get
	\begin{eqnarray*}
		&&(a_1-d_1)xP(yP(z))+(a_1-d_2)xP(P(y)z)+(\lambda a_1-d_3)xP(yz)+(\kappa a_1+ c_3-\kappa e_2-f_3)xyz\\
		&&+(a_2-e_2)P(xP(y))z+ (a_3-f_2)xP(y)z
		+(b_1-e_1)P(x)yP(z)+(b_2-e_2)P(P(x)y)z\\
		&&+(b_3-e_3)P(x)yz
		+(c_1-f_1)xyP(z)+ (c_2-\lambda e_2)P(xy)z
		=0.
	\end{eqnarray*}

Now take $(R,P)$ to be the free \wmrb $(\sha^{\rm NC}_\frake(T^+(M)),\pe)$ on the vector space $M=\bfk\{x,y,z\}$ given in Corollary~\mref{coro:freerbm}. Then with the notation $\pe(u)=\lc u\rf$, the above equation becomes
	\begin{eqnarray*}
		&&(a_1-d_1)x\lf y \lf z\rf\rf+(a_1-d_2)x\lf \lf y\rf z\rf+(\lambda a_1-d_3)x\lf yz\rf+(\kappa a_1+ c_3-\kappa e_2-f_3)xyz\\
		&&+(a_2-e_2)\lf x\lf y\rf\rf z+ (a_3-f_2)x\lf y\rf z
		+(b_1-e_1)\lf x\rf y\lf z\rf +(b_2-e_2)\lf \lf x\rf y\rf z\\
		&&+ (b_3-e_3)\lf x\rf yz
		+ (c_1-f_1)xy\lf z\rf + (c_2-\lambda e_2)\lf xy\rf z
		=0.
	\end{eqnarray*}
Note that the bracketed words appearing in the linear combination: 
	$$x\lf y \lf z\rf\rf, x\lf \lf y\rf z\rf, x\lf yz\rf, xyz, \lc x\lc y\rc z\rc, x\lf y\rf z, \lf x\rf y\lf z\rf, \lf \lf x\rf y\rf z, \lf x\rf yz, xy\lf z\rf, \lf xy\rf z $$
are RBWs, which form part of the basis $\frakX_\infty$. Hence these bracketed words are linearly independent. Thus all coefficients in the above equation must be zero, giving a system of linear equations. Solving this system of linear equations yields
	\begin{eqnarray*}
		&a_1=d_1=d_2,\quad d_3=\lambda a_1, \quad 
		a_2=b_2=e_2, \quad c_2=\lambda a_2,\quad a_3=f_2, &\\
		& b_1=e_1, \quad b_3=e_3,\quad c_1=f_1,\quad f_3=\kappa a_1+c_3-\kappa a_2,&
	\end{eqnarray*}
	where $a_1,a_2,a_3,b_1,b_3,c_1,c_3$ are free variables.
	Substituting these equalities into the general relation $r$ in Eq.~\meqref{eq:rrel}, we find that any relation $r$ that can be satisfied by $\prec_P,\succ_P,\obp$ for all \wmrbs $(R,P)$, in particular for $(\sha^{\rm NC}_{\frakc}(\bfk\{x,y,z\},P_{\frakc})$, must be of the form
	\begin{eqnarray*}
		r&=&a_1\Big((\prec\ot \prec,0)+(0,\prec\ot \prec)+(0,\prec\ot \succ)+\lambda(0,\prec\ot \ob)+{\kappa}(0,\ob\ot\ob)\Big)\\
		&&+a_2\Big((\prec\ot\succ,0)+(\succ\ot \succ,0)+\lambda(\ob\,\ot\succ,0)+\kappa(\ob\ot\ob,0)+(0,\succ\ot \succ)\Big)\\
		&&+a_3\Big((\prec\ot \ob,0)+(0,\ob\ot\succ)\Big)\\
		&&+b_1\Big((\succ\ot \prec,0)+(0,\succ\ot \prec)\Big)\\
		&&+b_3\Big((\succ\ot\ob,0)+(0,\succ\ot\ob)\Big)\\
		&&+c_1\Big((\ob\,\ot \prec,0)+(0,\ob\ot\prec)\Big)\\
		&&+(c_3-a_2\kappa)\Big((\ob\ot\ob,0)+(0,\ob\ot\ob)\Big).
	\end{eqnarray*}
Consequently, $r$ is in $\dfrel_{\rm ETD}$, proving $\ker \Phi \cap (V^{\ot 2}\oplus V^{\ot 2})\subseteq \dfrel_{\rm ETD}$.
	\smallskip
	
Conversely, we have shown in Proposition~\mref{prop:eda}, that the relations listed in Eqs.~\meqref{eq:e1}--\meqref{eq:e7}, namely the elements of $\dfrel_{\rm ETD}$, are satisfied by all extended Rota-Baxter algebras $(R,P)$ of weight $(\lambda,\kappa)$. Hence $\dfrel_{\rm ETD}$ is contained in $\ker \Phi\cap (V^{\ot 2}\oplus V^{\ot 2})$. This completes the proof. 
\end{proof}

\subsection{Extended dendriform algebras as the di-companion of the extended Rota-Baxter algebras} 
\mlabel{ss:ederb}
From extended Rota-Baxter algebras, we now give another derived structure that has two binary operations instead of three as in the case of the extended tridendriform algebra. Here the situation is more complicated that requires a less direct definition of an extended dendriform algebra of weight $(\lambda,\kappa)$. 

\begin{defn} Let $V=\bfk\{\prec,\succ\}$ be the vector space with basis $\{\prec,\succ\}$ with $\star=\prec+\succ$ and let $\lambda,\kappa\in\bfk$. Define
\begin{equation}		
\dfrel_{\rm ED}:=\dfrel_{\rm ED,\lambda,\kappa}:=\bfk\left\{
\begin{array}{l}
	a(\prec\ot \prec,\prec\ot\star)+b(\succ\ot \prec,\succ\ot\prec)\\
	+c(\star\ot\succ,\succ\ot \succ) \end{array} \left | 
\begin{array}{l} 
\lambda a=\lambda c=0, \kappa(a-c)=0, \\
a,b,c\in\bfk\end{array}\right.
\right \}
\mlabel{eq:edrelation}
\end{equation}
The quotient operad $\oed:=\oed_{\lambda,\kappa}:=\calp(V,\dfrel_{\rm ED,\lambda,\kappa})$ is called the {\bf operad of extended dendriform algebras of weight $(\lambda,\kappa)$}.
Then an $\oed_{\lambda,\kappa}$-algebra is called an {\bf extended dendriform algebra of weight $(\lambda,\kappa)$}.
\end{defn}

We now show that the operad of extended dendriform algebras is the di-companion of the opeard of extended Rota-Baxter algebra. 

\begin{theorem}
Fix $\lambda,\kappa\in \bfk$. The operad $\oed$ is the di-companion $\mathrm{DC}(\oerb)$ of the operad of extended Rota-Baxter algebras. 
More precisely, let $V=\bfk\{\prec,\succ\}$ be the vector space with basis $\{\prec,\succ\}$. For the operadic homomorphism 
$$ \Psi: \calf(V) \longrightarrow \oerb, \quad	\prec\,\mapsto \mu(\id\ot \rho), \quad \succ\,\mapsto \mu(\rho\ot \id),$$
we have $\ker \Psi \cap (V^{\ot 2}\oplus V^{\ot 2})= \dfrel_{\rm ED}$.
	\mlabel{thm:edn}
\end{theorem}

\begin{proof}
An element $r$ of $\dfgen^{\otimes 2} \oplus \dfgen^{\otimes 2}$ is of the form
	\begin{equation}
\begin{split}		
	r=&a_1(\prec\ot \prec,0)+a_2(\prec\ot\succ,0)+b_1(\succ\ot \prec,0)+b_2(\succ\ot \succ,0)\\
		&+c_1(0,\prec\ot \prec)+c_2(0,\prec\ot \succ)+d_1(0,\succ\ot \prec)+d_2(0,\succ\ot \succ),
\end{split}
\mlabel{eq:rrell}
	\end{equation}
	where the coefficients $a_i,b_i,c_i,d_i$ are in $\bfk$ for $1\leq i\leq 2$.

If $r$ is in $\ker \Psi$, then for every extended Rota-Baxter algebra $(R,P)$, we have 
	\begin{eqnarray*}
		&&a_1xP(y)P(z)+a_2P(xP(y))z+b_1P(x)yP(z)+b_2P(P(x)y)z\\
		&&-c_1xP(yP(z))-c_2xP(P(y)z)-d_1P(x)yP(z)-d_2P(x)P(y)z
		=0.
	\end{eqnarray*}
Since $P$ is an \wmrbo, we further have
	\begin{eqnarray*}
		&&a_1xP(yP(z))+a_1xP(P(y)z)+\lambda a_1xP(yz)+\kappa a_1xyz+a_2P(xP(y))z\\
		&&+b_1P(x)yP(z)
		+b_2P(P(x)y)z-c_1xP(yP(z))-c_2xP(P(y)z)\\
	&&-d_1P(x)yP(z)-d_2P(xP(y))z-d_2P(P(x)y)z
		-\lambda d_2P(xy)z-\kappa d_2xyz
		=0.
	\end{eqnarray*}
Collecting like terms leads to
	\begin{eqnarray*}
		&&(a_1-c_1)xP(yP(z))+(a_1-c_2)xP(P(y)z)+\lambda a_1xP(yz)+(\kappa a_1-\kappa d_2)xyz\\
		&&+(a_2-d_2)P(xP(y))z
		+(b_1-d_1)P(x)yP(z)+(b_2-d_2)P(P(x)y)z-
	 \lambda d_2P(xy)z
		=0.
	\end{eqnarray*}

Take $(R,P)$ to be the free \wmrb $(\sha^{\rm NC}_\frake(T^+(M)),\pe)$ on the vector space $M=\bfk\{x,y,z\}$ given in Corollary~\mref{coro:freerbm}. Then we get
\begin{eqnarray*}
		&&(a_1-c_1)x\lc y\lc z\rc\rc+(a_1-c_2)x\lc\lc y\rc z\rc+\lambda a_1x\lc yz\rc+(\kappa a_1-\kappa d_2)xyz\\
		&&+(a_2-d_2)\lc x\lc y\rc\rc z
		+(b_1-d_1)\lc x\rc y\lc z\rc+(b_2-d_2)\lc\lc x\rc y\rc z-
	 \lambda d_2\lc xy\rc z
		=0.
	\end{eqnarray*}

Now the bracketed words occurring in the above equation: 
	$$x\lf y \lf z\rf\rf, x\lf \lf y\rf z\rf, x\lf yz\rf, xyz, \lc x\lc y\rc z\rc, x\lf y\rf z, \lf x\rf y\lf z\rf, \lf \lf x\rf y\rf z, \lf x\rf yz, xy\lf z\rf, \lf xy\rf z $$
are RBWs and hence are part of a basis of $\frakX_\infty$. Thus, they are linearly independent. So all coefficients in the above equation must be zero, yielding a system of linear equations: 
\begin{equation*}
\left\{\begin{array}{llll}
a_1-c_1=0,&
a_1-c_2=0,&
\lambda a_1=0,&
\kappa(a_1-d_2)=0,\\
a_2-d_2=0,&
b_2-d_2=0,&
b_1-d_1=0,&
\lambda d_2=0.
\end{array}\right.
\end{equation*}
Solving this system of linear equations, we obtain
\begin{equation}
\begin{split}
\left\{\begin{array}{lll}
		a_1=c_1=c_2,&\lambda a_1=0,& \kappa(a_1-d_2)=0,\\
		d_2=a_2=b_2, &
 b_1=d_1,&
\lambda d_2=0,
\end{array}\right.
\end{split}
\mlabel{eq:edrell}
\end{equation}
where $a_1,d_2,b_1$ are free variables. Substituting these equations into the general relation $r$ in Eq.~\meqref{eq:rrell} leads to $r\in \dfrel_{\rm ED}$.
This proves $\ker  \Psi\cap (V^{\ot 2}\oplus V^{\ot 2}) \subseteq \dfrel_{\rm ED}$. 
	
On the other hand, any $r\in \dfrel_{\rm ED}$ can be expressed as $$r=a(\succ\ot\succ,\succ\ot\star)+b(\succ\ot\prec,\prec\ot\succ)+c(\star\ot \succ,\succ\ot\succ),$$ 
where $\lambda a=\lambda c=0$ and $\kappa(a-c)=0$. By a direct computation, $\Psi(r)$ is satisfied by every \wmrb $(R,P)$ of weight $(\lambda,\kappa)$. Thus, $\dfrel_{\rm ED} \subseteq \ker \Psi \cap (V^{\ot 2}\oplus V^{\ot 2})$. This completes the proof. 
\end{proof}

By Theorem~\mref{thm:edn}, we obtain
\begin{coro}With the above notations, we have 
\begin{equation*}
\dfrel_{\rm ED,\lambda,\kappa}:=\left\{\begin{array}{ll}
\bfk\{(\prec \ot \prec,\prec\ot \star),\;(\succ \ot \prec,\succ\ot \prec), \;(\star\ot \succ, \succ\ot \succ)\}, &\text{if}\; \lambda=0$,\, $\kappa= 0,\\
\bfk\{(\prec \ot \prec+\star\ot \succ,\prec\ot \star+\succ\ot \succ),\,(\succ \ot \prec,\succ\ot \prec)\},&\text{if}\;\lambda=0,\,\kappa\neq 0,\\
\bfk\{(\succ \ot \prec,\succ\ot \prec)\},&\text{if}\; \lambda\neq0.\\
\end{array}\right.
\end{equation*}
\mlabel{co:threecases}
\end{coro}
\begin{proof} To solve Eq.~\meqref{eq:edrell}, we consider the following three cases.
\smallskip

\noindent
{\bf Case 1.} $\lambda=0$ and $\kappa= 0$. Then
$$
a_1=c_1=c_2,\;
d_2=a_2=b_2,\;
b_1=d_1.
$$
Thus, $\dfrel_{\rm ED,\lambda,\kappa}=\bfk\{(\prec \ot \prec,\prec\ot \star),\;(\succ \ot \prec,\succ\ot \prec), \;(\star\ot \succ, \succ\ot \succ)\}$.
\smallskip

\noindent
{\bf Case 2.} $\lambda=0$ and $\kappa\neq 0$. Then $a_1=d_2$. Thus
$$
a_1=c_1=c_2=d_2=a_2=b_2,\;
 b_1=d_1.
$$
So $\dfrel_{\rm ED,\lambda,\kappa}=\bfk\{(\prec \ot \prec+\star\ot \succ,\prec\ot \star+\succ\ot \succ),\,(\succ \ot \prec,\succ\ot \prec)\}$.
\smallskip

\noindent
{\bf Case 3.} $\lambda\neq 0$. Then
$$
a_1=c_1=c_2=d_2=a_2=b_2=0, \;
 b_1=d_1.
$$
Hence, $\dfrel_{\rm ED,\lambda,\kappa}=\bfk\{(\succ \ot \prec,\succ\ot \prec)\}$.
\end{proof}

In summary, depending on the vanishings of $\lambda$ and $\kappa$, there are three types of extended dendriform algebras as follows. 

\begin{defn}Let $\oed_{\lambda,\kappa}$ be the  operad of extended dendriform algebras of weight $(\lambda,\kappa)$.
 When $\lambda=0$ and $\kappa= 0$ (resp. $\lambda=0$ and $\kappa\neq 0$, resp. $\lambda\neq 0$), then the corresponding $\oed_{\lambda,\kappa}$-algebra is called a {\bf type I} (resp. {\bf II}, resp. {\bf III}) {\bf extended dendriform algebra of weight $(\lambda,\kappa)$}.
\mlabel{defn:type}
\end{defn}

\begin{remark}By Corollary~\mref{co:threecases},
we see that a type I extended dendriform algebra of weight $(\lambda,\kappa)$ is just a dendriform algebra. 
By  Definition~\mref{defn:edri}, a type II extended dendriform algebra  of weight $(\lambda,\kappa)$ is exactly an extended dendriform algebra.  Furthermore,  a type III extended dendriform algebra is precisely an $L$-algebra introduced in~\cite[Definition~1.3]{LP03}.
\end{remark}

\smallskip

\noindent {\bf Acknowledgements}: This work is supported by the National Natural Science Foundation of
China (12326324, 12461002) and Jiangxi Provincial Natural Science Foundation (20224BAB201 003). The first author thanks the Chern Institute of Mathematics at Nankai University for hospitality.

\noindent
{\bf Declaration of interests. } The authors have no conflicts of interest to disclose.

\noindent
{\bf Data availability. } Data sharing is not applicable as no new data were created or analyzed.

\end{document}